\documentclass[a4paper,10pt]{amsart}
\usepackage{amsmath,amsthm,amssymb,latexsym,enumerate,color,hyperref}
\usepackage{graphicx}
\usepackage[margin=2.2cm]{geometry}
\usepackage{xcolor}

\renewcommand{\le}{\leqslant}

\renewcommand{\ge}{\geqslant}

\makeatletter
\@namedef{subjclassname@2020}{\textup{2020} Mathematics Subject Classification}
\makeatother

\numberwithin{equation}{section}

\begin{document}

\newtheorem{thm}{Theorem}[section]
\newtheorem{prop}[thm]{Proposition}
\newtheorem{lem}[thm]{Lemma}
\newtheorem{cor}[thm]{Corollary}
\newtheorem{rem}[thm]{Remark}
\newtheorem*{defn}{Definition}

\newcommand{\DD}{\mathbb{D}}
\newcommand{\NN}{\mathbb{N}}
\newcommand{\ZZ}{\mathbb{Z}}
\newcommand{\QQ}{\mathbb{Q}}
\newcommand{\RR}{\mathbb{R}}
\newcommand{\CC}{\mathbb{C}}
\renewcommand{\SS}{\mathbb{S}}

\renewcommand{\theequation}{\arabic{section}.\arabic{equation}}

\newcommand{\supp}{\mathop{\mathrm{supp}}}    

\newcommand{\re}{\mathop{\mathrm{Re}}}   
\newcommand{\im}{\mathop{\mathrm{Im}}}   
\newcommand{\dist}{\mathop{\mathrm{dist}}}  
\newcommand{\link}{\mathop{\circ\kern-.35em -}}
\newcommand{\spn}{\mathop{\mathrm{span}}}   
\newcommand{\ind}{\mathop{\mathrm{ind}}}   
\newcommand{\rank}{\mathop{\mathrm{rank}}}   
\newcommand{\Fix}{\mathop{\mathrm{Fix}}}   
\newcommand{\codim}{\mathop{\mathrm{codim}}}   
\newcommand{\conv}{\mathop{\mathrm{conv}}}   
\newcommand{\epsi}{\mbox{$\varepsilon$}}
\newcommand{\eps}{\mathchoice{\epsi}{\epsi}
{\mbox{\scriptsize\epsi}}{\mbox{\tiny\epsi}}}
\newcommand{\cl}{\overline}
\newcommand{\pa}{\partial}
\newcommand{\ve}{\varepsilon}
\newcommand{\zi}{\zeta}
\newcommand{\Si}{\Sigma}
\newcommand{\cA}{{\mathcal A}}
\newcommand{\cG}{{\mathcal G}}
\newcommand{\cH}{{\mathcal H}}
\newcommand{\cI}{{\mathcal I}}
\newcommand{\cJ}{{\mathcal J}}
\newcommand{\cK}{{\mathcal K}}
\newcommand{\cL}{{\mathcal L}}
\newcommand{\cN}{{\mathcal N}}
\newcommand{\cR}{{\mathcal R}}
\newcommand{\cS}{{\mathcal S}}
\newcommand{\cT}{{\mathcal T}}
\newcommand{\cU}{{\mathcal U}}
\newcommand{\OM}{\Omega}
\newcommand{\B}{\bullet}
\newcommand{\ol}{\overline}
\newcommand{\ul}{\underline}
\newcommand{\vp}{\varphi}
\newcommand{\AC}{\mathop{\mathrm{AC}}}   
\newcommand{\Lip}{\mathop{\mathrm{Lip}}}   
\newcommand{\es}{\mathop{\mathrm{esssup}}}   
\newcommand{\les}{\mathop{\mathrm{les}}}   
\newcommand{\nid}{\noindent}
\newcommand{\pzr}{\phi^0_R}
\newcommand{\pir}{\phi^\infty_R}
\newcommand{\psr}{\phi^*_R}
\newcommand{\pow}{\frac{N}{N-1}}
\newcommand{\ncl}{\mathop{\mathrm{nc-lim}}}   
\newcommand{\nvl}{\mathop{\mathrm{nv-lim}}}  
\newcommand{\la}{\lambda}
\newcommand{\La}{\Lambda}    
\newcommand{\de}{\delta}    
\newcommand{\fhi}{\varphi} 
\newcommand{\ga}{\gamma}    
\newcommand{\ka}{\kappa}   

\newcommand{\core}{\heartsuit}
\newcommand{\diam}{\mathrm{diam}}

\newcommand{\lan}{\langle}
\newcommand{\ran}{\rangle}
\newcommand{\tr}{\mathop{\mathrm{tr}}}
\newcommand{\diag}{\mathop{\mathrm{diag}}}
\newcommand{\dv}{\mathop{\mathrm{div}}}

\newcommand{\al}{\alpha}
\newcommand{\be}{\beta}
\newcommand{\Om}{\Omega}
\newcommand{\na}{\nabla}

\newcommand{\cC}{\mathcal{C}}
\newcommand{\cM}{\mathcal{M}}
\newcommand{\nr}{\Vert}
\newcommand{\De}{\Delta}
\newcommand{\cX}{\mathcal{X}}
\newcommand{\cP}{\mathcal{P}}
\newcommand{\om}{\omega}
\newcommand{\si}{\sigma}
\newcommand{\te}{\theta}
\newcommand{\Ga}{\Gamma}

\title[Anisotropic radial symmetry]{Radial symmetry of solutions \\
to anisotropic and weighted diffusion equations \\
with discontinuous nonlinearities}

\author{Serena Dipierro}
\address{Serena Dipierro: Department of Mathematics and Statistics, The University of Western Australia, 35 Stirling Highway, Crawley, Perth, WA 6009, Australia}
\email{serena.dipierro@uwa.edu.au}

\author{Giorgio Poggesi}
\address{Giorgio Poggesi: Department of Mathematics and Statistics, The University of Western Australia, 35 Stirling Highway, Crawley, Perth, WA 6009, Australia}
\email{giorgio.poggesi@uwa.edu.au}

\author{Enrico Valdinoci}
\address{Enrico Valdinoci: Department of Mathematics and Statistics, The University of Western Australia, 35 Stirling Highway, Crawley, Perth, WA 6009, Australia}
\email{enrico.valdinoci@uwa.edu.au}

\begin{abstract}
For $1 <p < \infty$, we prove radial symmetry for bounded nonnegative solutions of
\begin{equation*}
	\begin{cases}
		-\dv\left\{ w(x) \, H( \na u )^{p-1}\, \na_{ \xi} H( \na u) \right\}= f(u) \, w(x) \ & \mbox{ in } \ \Si \cap \Om, 
		\\
		u=0 \ & \mbox{ on } \ \Ga_0 ,
		\\
		\langle \na_\xi H(\na u) , \nu \rangle = 0 \ & \mbox{ on } \ \Ga_1 \setminus \left\lbrace 0 \right\rbrace ,
	\end{cases}
\end{equation*}
where $\Om$ is a Wulff ball, $\Si$ is a convex cone with vertex at the center of $\Om$, 
$\Ga_0 := \Si \cap \pa \Om$, $\Ga_1 := \pa \Si \cap \Om $, $H$ is a norm,
$w$ is
a given weight and~$f$ is a possibly discontinuous nonnegative nonlinearity.

Given the anisotropic setting that we deal with,
the term ``radial'' is understood in the Finsler framework, that is, the function~$u$ is radial if there exists a point~$x$ such that $u$ is constant on the Wulff shapes centered at~$x$.

When $\Si = \RR^N$, J. Serra obtained the symmetry result in the isotropic unweighted setting
(i.e., when~$H(\xi)\equiv|\xi|$ and~$w\equiv1$). In this case we provide the extension of his result to the anisotropic setting. This provides a generalization to the anisotropic setting of a celebrated result due to Gidas-Ni-Nirenberg and such a generalization is new even for $p=2$ whenever $N>2$.

When $\Si \subsetneq \RR^N$ the results presented are new even in the isotropic and unweighted setting (i.e., when $H$ is the Euclidean norm and $w \equiv 1$) whenever $2 \neq p \neq N$.
Even for the previously known case of
unweighted isotropic setting
with~$p=2$ and $\Si \subsetneq \RR^N$,
the present paper provides an approach to the problem by exploiting integral (in)equalities which is new for $N>2$: this complements the corresponding symmetry result obtained via the moving planes method by Berestycki-Pacella.

The results obtained in the isotropic and weighted setting (i.e., with $w \not\equiv 1$)
are new for any $p$.
\end{abstract}

\keywords{Symmetry, Convex cones, Weighted anisotropic isoperimetric inequalities}
\subjclass[2020]{35B06}
%
%

\maketitle

\raggedbottom

\section{Introduction}

Since the classical works of
Aleksandrov~\cite{A}, Serrin~\cite{S} and
Gidas--Ni--Nirenberg~\cite{GNN}, an intensively studied topic
in the theory of partial differential equations and the calculus of variations focuses
on the radial symmetry of solutions under suitable assumptions
on the equation under consideration, on the boundary conditions and/or on the domain.

In this paper we consider the Finsler framework of a nonlinear anisotropic equation in divergence form.
The setting taken into account is quite general, since it includes nonlinear operators
of $p$-Laplace type, possibly with weights. The domain considered is obtained
from the intersection of a Wulff ball and a cone~$\Sigma \subseteq \RR^N$ with vertex at its center.
The
results provided are of radial symmetry type, where, given
the possible anisotropy of the ambient space,
the term ``radial'' is understood in the sense that a function~$u$ is radial
if there exists a point~$x$ such that $u$ is constant on the Wulff shapes 
centered at~$x$ (see Section~\ref{sec:preliminaries} for the detailed mathematical setting).

These results
will be obtained
under suitable
homogeneity or concavity assumptions on the weights~$w$ and suitable arithmetic 
relations between the dimension of the ambient space, the homogeneous exponent of the weight and
the homogeneity of the nonlinear operator.
A pivotal step in our analysis consists in proving that the level sets of the solution
are isoperimetric (and additionally the anisotropic norm of the gradient of the solution is constant
along these level sets). This cornerstone result will be stated in detail in Theorem~\ref{thm:MAIN GENERAL}
below and then combined with a series of isoperimetric inequalities to establish 
that the level sets are Wulff shapes: this, together with the
additional information
on the constancy of the anisotropic norm of the gradient of the solution, establishes
the radial symmetry of the solution. Depending on technical conditions on the (an)isotropy
of the ambient space and on the weights, these radial symmetry results
will be detailed in Theorems~\ref{thm:cones},
\ref{thm:MAIN unweighted in Wulff ball}
and~\ref{thm:weighted cones}.

The arguments that we use are inspired by those of J. Serra
in~\cite{Se}, which in turn were inspired by the classical paper of P.-L. Lions~\cite{L} in which the symmetry was obtained in the Euclidean case with $p=N=2$, $\Si = \RR^N$ and~$w \equiv 1$.
We remark that the main results in~\cite{L}, focusing on elliptic semilinear equations
in the plane, can be seen as a counterpart of those in~\cite{GNN}, in the sense that
the results in~\cite{L} weaken the smoothness assumptions on the source term
and on the solution with respect to the setting in~\cite{GNN},
at the expense of restricting to positive nonlinearities and to dimension two only.
Interestingly, the method in~\cite{L}, being based on integral (in)equalities,
is conceptually different from the moving plane technique used in~\cite{GNN}
and, in a sense, it is more related, apart from several important structural differences,
to the approach to radial symmetry inhaugurated by Weinberger in \cite{W}:
see also, e.g., \cite{PSc, Re, GL, FGK, FK, BNST, MP, MP2, MP3, GS, EP, Po, CS, WX, BC, GX, DPV, PT, PT2, CG, CiR, CFR} for related problems and ramifications.

The method of~\cite{L}
has been extended by
Kesavan--Pacella to the
$N$-dimensional Euclidean setting in presence of the $p$-Laplacian operator in the case $p=N$,
see~\cite{KP}.
In turn,
the techniques of~\cite{KP} have triggered the research in the anisotropic case (with $\Si = \RR^N$ and $p=N$),
which was carried out by Belloni--Ferone--Kawohl
in~\cite{BFK}. 
The extension of the method to the case $p \neq N$ is due to J. Serra (\cite{Se}), who obtained it in the isotropic setting (with $\Si = \RR^N$).

In \cite{BP}, by using the moving planes method, Berestycki--Pacella obtained radial symmetry in spherical convex cones (in the isotropic unweighted setting), when $p=2$ and $f$ is Lipschitz.
To the authors' knowledge, symmetry results in spherical convex cones with $\Si \subsetneq \RR^N$ were considered only in the unweighted isotropic setting in \cite{BP} (for $p=2$) and \cite{KP} (for $p=N$).

See also~\cite{DFM} and the references therein for further works on symmetry problems
for nonlinear equations.\bigskip

We point out that, while our method is general enough to work in the anisotropic weighted setting,
the results offered here are new also in some more classical cases (such as those stated in Theorems~\ref{thm:cones}, \ref{thm:MAIN unweighted in Wulff ball}
and~\ref{thm:weighted cones}).
Moreover, thanks to the available characterizations of
the isoperimetric sets, the conclusions obtained in these
cases are stronger and take a less technical form. As an introductory example, one of the results that we obtain here (see Theorem~\ref{thm:MAIN unweighted in Wulff ball})
goes as follows. \bigskip

{\em
	Let $\Om$ be a Wulff ball in $\RR^N$, $N \ge 2$, and let $1 < p < \infty $. Assume that
	$f \in L^\infty_{\mathrm{loc}} ( \left[0, \infty \right) )$ is nonnegative.
	Let $u \in 
	C^1( \ol{ \Om } )$
	be a weak solution of 
	\begin{equation*}
	\begin{cases}
		-\dv\left\{   H( \na u )^{p-1}\, \na_{ \xi} H( \na u) \right\}= f(u) \ & \mbox{ in } \ \Om, 
		\\
				u \ge 0 &  \mbox{ in } \ \Om , \\ 
		u=0 \ & \mbox{ on } \ \pa \Om ,
	\end{cases}
\end{equation*}
	Assume that either 
	\begin{equation*}
		p \ge N ,
	\end{equation*}
	or
	\begin{equation*}
		p < N \quad \text{and, for some nonincreasing function $\phi \ge 0$, we have } \phi \le f \le \frac{N p}{N-p} \phi.
	\end{equation*}
	
	Then, $u$ is a radially symmetric and radially nonincreasing function. 
	Moreover, 
	\begin{equation*}
		u  \text{ is radially strictly decreasing on } \left\lbrace 0 < u < \max_{\ol{\Om}} u \right\rbrace  ,
	\end{equation*}
	and
	\begin{equation*}
	\left\lbrace 0 < u < \max_{ \ol{\Om} } u \right\rbrace \quad \text{is a Wulff annulus or a punctured Wulff ball.}
	\end{equation*}}

\bigskip

Let us now go into the technical framework of this paper, to present the results obtained in their full generality.

Let $\Si$ be an open cone in $\RR^N$ with vertex at the origin, i.e.,
\begin{equation*}
	\Si := \left\lbrace t x \, : \, x \in \om , \, t \in (0, \infty ) \right\rbrace ,
\end{equation*}
for some domain $\om \subseteq \SS^{N-1}$.
%
%
We stress that the possibility $\om = \SS^{N-1}$ (which implies $\Si= \RR^N$) is allowed throughout the paper.
When~$\om \subsetneq \SS^{N-1}$,
we assume $\Sigma$ to be convex and denote by $\nu$ its outward unit normal (which is defined almost everywhere on $\pa \Si$).
%
%

Let $\Om \subset \RR^N$ be a bounded domain, and define
\begin{equation*}
	\Ga_0 := \Si \cap \pa \Om  \quad{\mbox{ and }} \quad 
	\Ga_1 := \pa \Si \cap \Om .
\end{equation*}

Furthermore, we endow $\RR^N$ with a norm $H : \RR^N \to \RR$
such that:
\begin{eqnarray}
&&{\mbox{$H$ is convex;}}\label{Hcond1}\\
&&{\mbox{$H( \xi) \ge 0$ for $\xi \in \RR^N$ and $H(\xi) = 0$ if and only if $\xi = 0$;}}\label{Hcond2}\\
&&{\mbox{$H(t \xi) = |t|H(\xi)$ for $\xi \in \RR^N$ and $t \in \RR$.}}\label{Hcond3}
\end{eqnarray}
We also define
\begin{equation}\label{ball1}
B^H := \left\lbrace \xi \in \RR^N : H( \xi ) < 1 \right\rbrace.
\end{equation}
Throughout the paper, we assume $H \in C^2 ( \RR^N \setminus \left\lbrace 0 \right\rbrace )$ to be uniformly elliptic.
We say that $H \in C^2 ( \RR^N \setminus \left\lbrace 0 \right\rbrace )$ is uniformly elliptic if the ball $B^H$ is uniformly convex, i.e., such that the principal curvatures of the boundary of $B^H$ are bounded away from zero. This is a standard assumption on $H$ in order to obtain some regularity of the solutions, using or adapting the elliptic regularity theory (see, e.g., also \cite{BC, CFV}).

In this setting, we consider solutions of
\begin{equation}\label{eq:GENERALPB weighted anisotropic problem in cones}
	\begin{cases}
		-\dv\left\{ w(x) \, H( \na u )^{p-1}\, \na_{ \xi} H( \na u) \right\}= f(u) \, w(x) \ & \mbox{ in } \ \Si \cap \Om, 
		\\
		u \ge 0 & \mbox{ in } \ \Si \cap \Om, \\ 
		u=0 \ & \mbox{ on } \ \Ga_0 ,
		\\
		\langle \na_\xi H(\na u) , \nu \rangle = 0 \ & \mbox{ on } \ \Ga_1 
		\setminus \left\lbrace 0 \right\rbrace.
	\end{cases}
\end{equation}
In the case $\Si = \RR^N$, we have $\Ga_1 = \varnothing$ and the last condition in \eqref{eq:GENERALPB weighted anisotropic problem in cones} is trivially satisfied. In this case, \eqref{eq:GENERALPB weighted anisotropic problem in cones} simply becomes
\begin{equation}
	\begin{cases}
	\label{eq:weighted anisotropic problem in RN}
	-\dv\left\{ w(x) \, H( \na u )^{p-1}\, \na_{ \xi} H( \na u) \right\}= f(u) \, w(x) \ & \mbox{ in } \ \Om, 
	\\
	u \ge 0 & \mbox{ in } \ \Om, 
	\\
	u=0 \ & \mbox{ on } \ \pa \Om .
	\end{cases}
\end{equation}

In our main results, the assumptions on the weight $w$ are the following\footnote{Using the recent results
in~\cite{I}, assumption~\eqref{eq:assumptions on w for main theorems2} may be weakened by requiring
that, when~$\la>0$, either~$w$ or~$w^{1/ \la }$ is concave,
or, equivalently, that either~$\la\in(0,1)$ and~$w$ is concave or~$\la\in[1,+\infty)$ and~$w^{1/ \la }$ is concave.
This generalization is accomplished by using, when~$w$ is concave, the isoperimetric
inequality in Corollary~0.14 (see also Remark~0.15) of~\cite{I} instead
of the one by~\cite{CRS} recalled here in Theorem~\ref{thm:isoperimetric anisotropic weighted cones}.
Strictly speaking, the setting in~\cite{I} is isotropic, but
the arguments introduced there have the potential to be generalized to obtain analogous results in the anisotropic setting as well.
Counterexamples
to the isoperimetric
inequality are also given in Corollary~0.12 and Remark~0.13 of~\cite{I}.

We also point out that
the constant
in~\cite{I} (which is in terms of the volume ratio) improves
in several cases the previously obtained ones.}: there exists
$\la \ge 0$ such that
\begin{eqnarray}\label{eq:assumptions on w for main theorems}
&&w :  \ol{\Si} \to \left[ 0, +\infty \right) \text{ is a continuous nonnegative function positively homogeneous of degree~$\lambda$,} 
\\
&&w^{1/ \la } \text{ is concave in } \Si \text{ (in the case $\la >0$)},\label{eq:assumptions on w for main theorems2}
\\&&
w \text{ is locally Lipschitz in } \Si . \label{eq:assumptions on w for main theorems3}
\end{eqnarray}
We notice 
that when $\la>0$, the concavity assumption on $w^{1 / \la}$ automatically gives that~$w$
is positive in~$\Si$.

To state our main theorem in this setting, we recall the following notation.
We identify the dual space of $\RR^N$ with $\RR^N$ itself via the scalar product
$\langle \cdot, \cdot \rangle$. Accordingly, given the norm~$H$ satisfying~\eqref{Hcond1}, \eqref{Hcond2}
and~\eqref{Hcond3},
 the space $\RR^N$ turns out to be endowed with the dual norm $H_0$, that is the polar function defined by
\begin{equation}\label{hzero}
	H_0(x) = \sup_{\xi \neq 0 } \frac{\langle x, \xi \rangle }{H( \xi )} \quad \text{for} \quad x \in \RR^N .
\end{equation} 
Notice that $H$ results to be the support function of the unitary Wulff ball 
\begin{equation}\label{ball33}
B^{H_0} := \left\lbrace x \in \RR^N : H_0(x) < 1 \right\rbrace \end{equation}
of $H$ centered at the origin (see \cite{CrM} and \cite[Section 1.7]{Sc}) and, in turn, $H_0$ is the
support function of~$B^H$, defined in~\eqref{ball1}. To ease notation, the unitary Wulff ball (of $H$) $B^{H_0}$
centered at the origin will be denoted simply by $B$.
Similarly, for $r>0$, $B_r$ will denote the Wulff ball (of $H$) of radius $r$ centered at the origin, i.e.,
$$
B_r := B_r^{H_0} := \left\lbrace x \in \RR^N : H_0(x) < r \right\rbrace .
$$

Furthermore, for a set of finite perimeter $E$, we define the weighted anisotropic perimeter of $E$ in $\Si$ as follows
\begin{equation}\label{eq:defweighted perimetro smooth}
P_{w,H} ( E ; \Si ) : = \int_{\Si \cap \pa^* E} H( \nu ) \, w \, d \cH^{N-1},
\end{equation}
being~$\pa^* E$ the reduced boundary of~$E$ and $\nu$ its outer (measure theoretical) unit normal vector. 
Also, for a measurable set $E \subset \Si$, we denote by $w( E )$ the weighted volume of $E$, namely
\begin{equation}\label{WEIVOL}
w( E ) : = \int_E w \, d \cH^{N} .
\end{equation}
In the unweighted case (i.e., when $w \equiv 1$), $w(E)$ agrees with $\cH^N ( E)$.

With this notation, our main theorem is the following:

\begin{thm}\label{thm:MAIN GENERAL}
	Let $\Om$ be a Wulff ball in $\RR^N$ centered at $0$, $N \ge 2$, and let $1 < p < \infty $. Assume that
	$f \in L^\infty_{\mathrm{loc}} ( \left[0, \infty \right) )$ is nonnegative
	and the weight $w$ satisfies \eqref{eq:assumptions on w for main theorems},
	\eqref{eq:assumptions on w for main theorems2}, and \eqref{eq:assumptions on w for main theorems3}.

	Let $u \in C^1 \left( (\Si \cap \Om) \cup  \Ga_0 \cup ( \Ga_1 \setminus \left\lbrace 0 \right\rbrace ) \right) \cap W^{1,\infty}(\Si \cap \Om )$
	be a solution of \eqref{eq:GENERALPB weighted anisotropic problem in cones} in
	the weak sense. 
	 Set 
	 \begin{equation}\label{Ddef}
	 D := N + \la,\end{equation} where $\la$ is that appearing in \eqref{eq:assumptions on w for main theorems}.
	Assume that either 
	\begin{equation}\tag{a}\label{eq:GENERAL condizione a}
		p \ge D ,
	\end{equation}
	or
	\begin{equation}\tag{b}\label{eq:GENERAL condizione b}
		p < D \quad \text{and, for some nonincreasing function $\phi \ge 0$, we have } \phi \le f \le \frac{D p}{D-p} \phi.
	\end{equation}
	Let also
\begin{equation}\label{eq:def M in Main thm}
		M:=\sup_{ \Si \cap \Om } u ,
\end{equation}
		
		Then, for a.e. $t \in (0,M )$	
		the following two conditions are verified:
		
		(i) $\left\lbrace u >t \right\rbrace$ satisfies
			\begin{equation}\label{eq:weighted anisotropic isoperimetric in conesEQ}
		\frac{P_{w,H} \big(\left\lbrace u >t \right\rbrace ; \Si\big)}{w\big( \Si \cap \left\lbrace u >t \right\rbrace
		\big)^{\frac{ D -1}{ D }}} = \frac{P_{w,H} (B ; \Si)}{w( \Si \cap B)^{\frac{ D -1 }{ D }}} ,
	\end{equation}
		(ii) $H( \na u)$ is constant on $\left\lbrace u = t \right\rbrace$.
\end{thm}

We remark that, in the setting of Theorem~\ref{thm:MAIN GENERAL},
since $u \in W^{1, \infty}(\Si \cap \Om)$
	and $\Si \cap \Om$ is convex, the
	Sobolev embedding theorem (see also \cite[Theorem 4.1]{H}) ensures that~$
	u \in C^{0,1}(\ol{\Si \cap \Om}) \subset C^{0}(\ol{\Si \cap \Om})$,
	and hence an equivalent definition in~\eqref{eq:def M in Main thm} is to set
$$	M:=\max_{\ol{\Si \cap \Om}} u .$$	

We also point out that formula~\eqref{eq:weighted anisotropic isoperimetric in conesEQ} says that
the sets~$\left\lbrace u >t \right\rbrace$ satisfy
the equality in the weighted anisotropic isoperimetric inequality in cones established in~\cite{CRS}. This will be recalled in
details in Theorem~\ref{thm:isoperimetric anisotropic weighted cones}.

\begin{rem}[On the regularity of the solutions]\label{rem:on the regualrity}
	{\rm
Concerning the regularity assumptions on $u$ in Theorem~\ref{thm:MAIN GENERAL}, we point out that the main hypothesis is that $u$ belongs to~$W^{1,\infty}(\Si \cap \Om )$, while the additional smoothness assumptions may be dropped. Moreover, when $\Si = \RR^N$, also the assumption $u \in W^{1,\infty}(\Si \cap \Om )$ could be dropped and replaced just by the assumption that $u$ is bounded.

Indeed, in the case where $\Si=\RR^N$, whenever $\Om$ is a $C^{1, \ga}$ domain, with~$0 < \ga \le 1$, then \cite[Theorem 1]{Li} ensures that bounded solutions to \eqref{eq:weighted anisotropic problem in RN} are automatically~$C^{1,\ga}(\ol{\Om})$, for some~$0 < \ga \le 1$. Notice that, our regularity assumption on $H$ ensures that $\Om$ in Theorem \ref{thm:MAIN GENERAL}, which is a Wulff ball, is of class $C^2$. 
Thus, in the case when $\Si=\RR^N$, all the regularity assumptions on $u$ in Theorem \ref{thm:MAIN GENERAL} are satisfied by any bounded solution.  

When $\Si \subsetneq \RR^N$, we have to consider the mixed boundary value problem \eqref{eq:GENERALPB weighted anisotropic problem in cones}. For such a problem, regularity up to the (whole) boundary is a delicate issue, and it strongly depends on how $\Si$ and $\Om$ intersect. 
				Nevertheless, \cite[Theorem 1]{Li} still ensures the $C^{1,\ga}$-regularity (of bounded solutions) up to $\Ga_0$; the $C^{1,\ga}$-regularity up to $\Ga_1 \setminus \left\lbrace 0 \right\rbrace$ could be obtained by \cite[Theorem 2]{Li}, if we further assume that $\Ga_1 \setminus \left\lbrace 0 \right\rbrace$ is of class $C^{1,\ga}$, with~$0 < \ga \le 1$, and $w$ satisfy the additional requirement
\begin{equation}
	w \text{ is positive and H\"older continuous on } \ol{ \Si \cap \Om} \setminus \left\lbrace 0 \right\rbrace .
\label{eq:assumptions on w for main theorems4}		
\end{equation}
We stress that our results are presented without these additional assumptions: that is, we do not require~\eqref{eq:assumptions on w for main theorems4} and we only ask~$\Si$ to be convex (with no need of the additional assumption that $\Ga_1 \setminus \left\lbrace 0 \right\rbrace$ is of class $C^{1, \ga}$). 
				
Essentially, the main regularity assumption that we impose on $u$ in Theorem \ref{thm:MAIN GENERAL} (when $\Si \subsetneq \RR^N$) is that~$u$ belongs to~$W^{1, \infty} ( \Si \cap \Om )$. This is needed in order to ensure the validity of certain integral identities, such as the Pohozaev-type identity \eqref{eq:weighted anisotropic Pohozaev}, by using an approximation argument.
We notice that, when $H$ is the Euclidean norm and $w \equiv 1$, \cite[Theorem 1.2]{CM} guarantees 
the~$W^{1, \infty}$
regularity up to $\Ga_1$ (including the vertex). As pointed out in \cite{PT}, the assumption $u \in W^{1, \infty} ( \Si \cap \Om )$ can be seen as a gluing condition between the cone and $\Ga_0$. Related to this, we mention that \cite[Proposition 6.1]{PT} guarantees $C^2 (\ol{\Om} \setminus \left\lbrace 0 \right\rbrace)$-regularity for weak solutions of
$$
\begin{cases}
- \De u =1 \ & \mbox{ in } \ \Si \cap \Om, 
\\
u=0 \ & \mbox{ on } \ \Ga_0 ,
\\
\langle \na u , \nu \rangle = 0 \ & \mbox{ on } \ \Ga_1 \setminus \left\lbrace 0 \right\rbrace ,
\end{cases}
$$
in general domains $\Om$, whenever $\Ga_0$ and $\Ga_1$ intersect orthogonally.
}	
\end{rem}

\begin{rem}[On the weights]
{\rm
The assumptions \eqref{eq:assumptions on w for main theorems}, \eqref{eq:assumptions on w for main theorems2} and \eqref{eq:assumptions on w for main theorems3} on $w$ guarantee the validity of the weighted anisotropic isoperimetric inequality \eqref{eq:weighted anisotropic isoperimetric in cones} in convex cones, obtained in \cite{CRS}.
The homogeneity assumption on $w$ contained in \eqref{eq:assumptions on w for main theorems} is also used to obtain the Pohozaev-type identity in Lemma \ref{lem:Pohozaev General}.

Various examples of weights satisfying \eqref{eq:assumptions on w for main theorems}, \eqref{eq:assumptions on w for main theorems2} and \eqref{eq:assumptions on w for main theorems3}
%
%
are provided in \cite{CRS}. 
%
%
We point out that all the weights provided in those examples also satisfy assumption \eqref{eq:assumptions on w for main theorems4}, up to substituting $\Si$ with a smaller cone $\Si' \subset \Si$ (see also item (iii) at page 2983 in \cite{CRS}). 
}
\end{rem}

As pointed out in \cite{CRS}, the equality holds in the weighted anisotropic isoperimetric
inequality in~\eqref{eq:weighted anisotropic isoperimetric in cones} whenever $\Si \cap E = \Si \cap B_r$, where $r$ is any positive number. That is, Wulff balls (centered at the origin) intersected with $\Si$ are always minimizers of the isoperimetric inequality \eqref{eq:weighted anisotropic isoperimetric in cones}. 
However, the uniqueness\footnote{Whenever a line of direction $a \in \RR^N$ is contained in $\Si$, \textit{uniqueness} will be up to translations in direction $a$ (see Section \ref{sec:isoperimetric inequalities}).} of those minimizers (i.e., the characterization of the equality sign in \eqref{eq:weighted anisotropic isoperimetric in cones}) in general is still not available in the literature (see Section \ref{sec:isoperimetric inequalities}).
Whenever the uniqueness of those minimizers is available, it is not difficult to obtain the radial symmetry of $u$, as a corollary of Theorem \ref{thm:MAIN GENERAL}.

The uniqueness of the minimizers can be obtained in the unweighted (i.e., when $w \equiv 1$) anisotropic setting by adapting the ideas in \cite{FI} to the anisotropic setting (see Theorem \ref{thm:characterization unweighted anisotropic}).
Notice that in the unweighted case we have $D=N$ (in Theorem \ref{thm:MAIN GENERAL}), and \eqref{eq:GENERALPB weighted anisotropic problem in cones} reads as follows:
\begin{equation}\label{eq:UNweighted problem in cones}
	\begin{cases}
		-\dv\left\{  H( \na u )^{p-1}\, \na_{ \xi} H( \na u) \right\}= f(u) \ & \mbox{ in } \ \Si \cap \Om, 
		\\
				u \ge 0 & \mbox{ in } \ \Si \cap \Om, \\ 
		u=0 \ & \mbox{ on } \ \Ga_0 ,
		\\
		\langle \na_\xi H(\na u) , \nu \rangle = 0 \ & \mbox{ on } \ \Ga_1 \setminus \left\lbrace 0 \right\rbrace .
	\end{cases}
\end{equation}

In this setting, the uniqueness of the minimizers of the isoperimetric inequality and Theorem \ref{thm:MAIN GENERAL} lead to

\begin{thm}[Symmetry in convex cones in the anisotropic unweighted setting]
	\label{thm:cones}
	Let $\Om$ be a Wulff ball in $\RR^N$ centered at $0$, $N \ge 2$, and let $1 < p < \infty $. Assume that
	$f \in L^\infty_{\mathrm{loc}} ( \left[0, \infty \right) )$ is nonnegative. Let $u \in C^1 \left( (\Si \cap \Om) \cup \Ga_0 \cup \Ga_1 \setminus \left\lbrace 0 \right\rbrace \right) \cap W^{1,\infty}(\Si \cap \Om )$ be a solution of \eqref{eq:UNweighted problem in cones} in
	the weak sense. Assume that either 
	\begin{equation}\tag{a$'$}\label{eq:condizione a}
		p \ge N ,
	\end{equation}
	or
	\begin{equation}\tag{b$'$}\label{eq:condizione b}
		p < N \quad \text{and, for some nonincreasing function $\phi \ge 0$, we have } \phi \le f \le \frac{N p}{N-p} \phi.
	\end{equation}
	
		Then, $u$ is a radially symmetric and radially nonincreasing function. 
		Moreover, by setting $M$ as in \eqref{eq:def M in Main thm},
		\begin{equation*}
			u \text{ is radially strictly decreasing on } \left\lbrace 0 < u < M \right\rbrace ,
		\end{equation*}
		and
		$$
		\left\lbrace 0 < u < M \right\rbrace \quad \text{is a Wulff annulus or a punctured Wulff ball centered at } 0 \text{ intersected with } \Si.
		$$
\end{thm}

We point out that conditions~\eqref{eq:condizione a} and~\eqref{eq:condizione b}
in Theorem~\ref{thm:cones} correspond, respectively, to conditions~\eqref{eq:GENERAL condizione a} and~\eqref{eq:GENERAL condizione b} of
Theorem~\ref{thm:MAIN GENERAL}
when $D=N$.

Also,
as usual, when $\Si = \RR^N$, we have $\Ga_1 = \varnothing$ and the last condition in \eqref{eq:UNweighted problem in cones} is trivially satisfied.
In this case, \eqref{eq:UNweighted problem in cones} simply becomes:
\begin{equation}
	\begin{cases}
		\label{eq:Unweighted problem in Wulff ball}
		-\dv\left\{   H( \na u )^{p-1}\, \na_{ \xi} H( \na u) \right\}= f(u) \ & \mbox{ in } \ \Om, 
		\\
				 u \ge 0 & \mbox{ in } \ \Om, \\ 
		u=0 \ & \mbox{ on } \ \pa \Om ,
	\end{cases}
\end{equation}
and Theorem \ref{thm:cones} reads as follows.

\begin{thm}[Symmetry in balls in the anisotropic unweighted setting]
	\label{thm:MAIN unweighted in Wulff ball}
	Let $\Om$ be a Wulff ball in $\RR^N$, $N \ge 2$, and let $1 < p < \infty $. Assume that
	$f \in L^\infty_{\mathrm{loc}} ( \left[0, \infty \right) )$ is nonnegative.
	Let $u \in 
	C^1( \ol{ \Om } )$
	be a solution of \eqref{eq:Unweighted problem in Wulff ball} in the weak sense. Assume that either \eqref{eq:condizione a} or \eqref{eq:condizione b} of Theorem \ref{thm:cones} holds true.
	
	Then, $u$ is a radially symmetric and radially nonincreasing function. 
	Moreover, 
	\begin{equation*}
		u  \text{ is radially strictly decreasing on } \left\lbrace 0 < u < \max_{\ol{\Om}} u \right\rbrace  ,
	\end{equation*}
	and
	\begin{equation*}
	\left\lbrace 0 < u < \max_{ \ol{\Om} } u \right\rbrace \quad \text{is a Wulff annulus or a punctured Wulff ball.}
	\end{equation*}
\end{thm}

For completeness, in light of Remark \ref{rem:on the regualrity},
we point out that the regularity assumption on the solution~$u$
taken in Theorem~\ref{thm:MAIN unweighted in Wulff ball} can be relaxed
by assuming only that the solution is
bounded.

\medskip

The uniqueness of the minimizers of the isoperimetric inequality is also available in the weighted isotropic setting (i.e., when $H$ is the Euclidean norm).
In this case, \eqref{eq:GENERALPB weighted anisotropic problem in cones} reads as follows
\begin{equation}\label{eq:Weighted problem in cones}
	\begin{cases}
		-\dv\left\{ w(x) \,  | \na u |^{p-2}\,  \na u \right\}= f(u) \, w(x)  \ & \mbox{ in } \ \Si \cap \Om, 
		\\
				 u \ge 0 & \mbox{ in } \ \Si \cap \Om, \\ 
		u=0 \ & \mbox{ on } \ \Ga_0 ,
		\\
		\langle \na u , \nu \rangle = 0 \ & \mbox{ on } \ \Ga_1 \setminus \left\lbrace 0 \right\rbrace .
	\end{cases}
\end{equation}

In this case, the uniqueness of the minimizers of the isoperimetric inequality and Theorem \ref{thm:MAIN GENERAL} lead\footnote{The
recent examples in~\cite{I} for isoperimetric
minimizers that
are half-balls instead of full balls
may also yield interesting statements
in terms of lack of symmetry of solutions to be compared to Theorem~\ref{thm:weighted cones} here.}
to

\begin{thm}[Symmetry in convex cones in the isotropic weighted setting]
\label{thm:weighted cones}
	Let $\Om$ be a (Euclidean) ball in $\RR^N$ centered at $0$, $N \ge 2$, and let $1 < p < \infty $. Assume that
	$f \in L^\infty_{\mathrm{loc}} ( \left[0, \infty \right) )$ is nonnegative
	and the weight $w$ satisfies \eqref{eq:assumptions on w for main theorems}, \eqref{eq:assumptions on w for main theorems2},
	and \eqref{eq:assumptions on w for main theorems3}.
	 Let $u \in C^1 \left( (\Si \cap \Om) \cup \Ga_0 \cup \Ga_1 \setminus \left\lbrace 0 \right\rbrace \right) \cap W^{1,\infty}(\Si \cap \Om )$ be a solution of \eqref{eq:Weighted problem in cones} in
	the weak sense.

	 Assume that either \eqref{eq:GENERAL condizione a} or \eqref{eq:GENERAL condizione b} of Theorem \ref{thm:MAIN GENERAL} holds true.
	
		Then, $u$ is a radially symmetric and radially nonincreasing function. 
		Moreover, by setting $M$ as in \eqref{eq:def M in Main thm},
		\begin{equation*}
			u \text{ is radially strictly decreasing on } \left\lbrace 0 < u < M \right\rbrace ,
		\end{equation*}
		and
		$$
		\left\lbrace 0 < u < M \right\rbrace \quad \text{is a annulus or a punctured ball centered at } 0 \text{ intersected with } \Si .
		$$
\end{thm}

We point out that
Theorem \ref{thm:MAIN unweighted in Wulff ball} provides a generalization to the anisotropic setting of the celebrated result contained in \cite{GNN} and such a generalization is new even for $p=2$ whenever $N>2$. We recall that \cite{GNN} exploits the moving planes method, which seems not to be helpful for anisotropic problems. 
Instead, the method via integral (in)equalities used here (and based on \cite{L})
turns out to be effective in this setting.

When $\Si \subsetneq \RR^N$, the results presented here are
new even in the isotropic and unweighted setting (i.e., when~$H$ is the Euclidean norm and $w \equiv 1$) whenever $ 2 \neq p \neq N$.
In the case $p=2$ and $\Si \subsetneq \RR^N, \, N>2$, the
symmetry result
was previously obtained only relying on the moving planes method (and in the unweighted isotropic setting) in \cite{BP}: even in this special case, the present paper provides a new approach to the problem via integral (in)equalities (based on \cite{L}) which complements \cite{BP}.

Also, the results obtained here in the isotropic and weighted setting (i.e., with $w \not\equiv 1$)
are new for any $p$.

\bigskip

Theorems \ref{thm:cones}, \ref{thm:MAIN unweighted in Wulff ball}
and \ref{thm:weighted cones} can be rephrased in terms of the variational formulation of \eqref{eq:GENERALPB weighted anisotropic problem in cones}, as follows. 
We restrict this formulation to the case when $f$ is a continuous nonlinearity to avoid problems with the differentiability of the functional.

\begin{thm}[Variational formulation of Theorems \ref{thm:cones}, \ref{thm:MAIN unweighted in Wulff ball} and \ref{thm:weighted cones}]
Let $\Om$ be a Wulff ball in $\RR^N$ centered at $0$, $N \ge 2$, and let $1 < p < \infty $. 
Assume that $f$ is a continuous nonnegative function defined on $\left[0, \infty \right)$.
Let
$$
u \in W^{1,\infty}_{\Ga_0}(\Si \cap \Om):= \left\lbrace v \in W^{1,\infty}(\Si \cap \Om) \, : \, v=0 \text{ on } \Ga_0 \right\rbrace
$$
be a nonnegative critical point of the functional
$$
\int_{\Si \cap \Om} \left\lbrace \frac{H^p ( \na u )}{p} - F(u) \right\rbrace w(x) \, d \cH^{N} , \quad \text{where } \, F(s):=\int_0^s f( \tau) \, d \tau .
$$
Assume that $u \in C^1 \left( (\Si \cap \Om) \cup \Ga_1 \setminus \left\lbrace 0 \right\rbrace \right)$ and that,
either 
\begin{equation*}
\begin{split}
& H \text{ is any uniformly elliptic norm of class } C^2 (\RR^N \setminus \left\lbrace 0 \right\rbrace) , \,  w \equiv 1 ,
\\
& \text{and either \eqref{eq:condizione a} or \eqref{eq:condizione b} of Theorem \ref{thm:cones} holds true, }
\end{split}
\end{equation*}
or
\begin{equation*}
\begin{split}
& H \text{ is the euclidean norm, } w \text{ is any weight satisfying \eqref{eq:assumptions on w for main theorems}, \eqref{eq:assumptions on w for main theorems2} and \eqref{eq:assumptions on w for main theorems3},}
\\
& \text{and either \eqref{eq:GENERAL condizione a} or \eqref{eq:GENERAL condizione b} of Theorem \ref{thm:MAIN GENERAL} holds true.}
\end{split}
\end{equation*}

Then, $u$ is a radially symmetric and radially nonincreasing function. Moreover, by setting $M$ as in \eqref{eq:def M in Main thm}, $u$ is radially strictly decreasing on the set $\left\lbrace 0 < u < M \right\rbrace$, which is a Wulff annulus or a punctured Wulff ball centered at $0$ intersected with $\Si$.
\end{thm}

We conclude this introduction
with a comment on the proofs of Theorems \ref{thm:cones}, \ref{thm:MAIN unweighted in Wulff ball}
and~\ref{thm:weighted cones}. As already mentioned, those theorems are obtained by putting together Theorem \ref{thm:MAIN GENERAL} and the characterization of the equality sign in the isoperimetric inequality \eqref{eq:weighted anisotropic isoperimetric in cones}.
When $\Si$ is a convex cone containing no lines, then the characterization of the minimizers of the isoperimetric inequality forces the superlevel sets to be Wulff balls all centered at $0$ (which is the vertex of $\Si$) intersected with $\Si$. This immediately gives the desired symmetry results. 
When $\Si$ contains lines instead,
%
%
the characterization of minimizers still informs us that all the superlevel sets are Wulff balls intersected with $\Si$, but a further step is needed in order to prove that all of those Wulff balls are centered at $0$. This step is accomplished in Lemma \ref{lem:conclusion in ball}. 
\bigskip

The rest of this paper is organized as follows.
Section~\ref{sec:preliminaries} recalls the basics of the Finsler
framework, placing the known results
into a setting convenient for the applications that we have in mind.

Other useful auxiliary results are collected in Section~\ref{INGREDIE},
which contains a suitable 
maximum principle in cones, a Stampacchia-type result,
some geometric differential inequalities and a 
Pohozaev-type identity.

In Section~\ref{sec:isoperimetric inequalities}
we deal with some isoperimetric inequalities, leveraging on
the existing literature~\cite{CGPRS, CR, CRS, FI, FM, FMP}
and adapting all the previous results
to the case under consideration in this paper.

The proofs of the main results in Theorems~\ref{thm:MAIN GENERAL},
\ref{thm:cones}, \ref{thm:MAIN unweighted in Wulff ball}
and \ref{thm:weighted cones} are then contained in Section~\ref{FIDIMO}.

\section{Preliminaries and setting}\label{sec:preliminaries}

Here we recall some basic facts on the Finsler framework.

Given the norm~$H$ satisfying~\eqref{Hcond1}, \eqref{Hcond2} and~\eqref{Hcond3} and the dual
norm~$H_0$ defined in~\eqref{hzero},
we observe that we can reconstruct $H$ in terms of $H_0$ as
\begin{equation}\label{eq:norma H def dual}
	H(\xi) = \sup_{ x \neq 0 } \frac{ \langle x, \xi \rangle }{H_0( x )} \quad \text{for} \quad \xi \in \RR^N .
\end{equation}
Consequently, it holds that
\begin{equation}\label{eq:Anisotropic_CS_dual}
	| \langle \xi , \eta \rangle | \le H(\xi) H_0 (\eta) \, , \quad \text{for any } \, \xi , \, \eta \in \RR^N .
\end{equation}

The two convex sets $B^{H_0}$ and $B^H$, defined in~\eqref{ball33} and~\eqref{ball1} respectively, are
both centrally symmetric and they are polar of each other.
We denote by $B_r^{H_0} (x)$ the ball centered at the point $x$ with radius $r$ in the
norm $H_0$, i.e.
$$
B_r^{H_0} (x) := \left\lbrace y \in \RR^N : H_0(x-y) < r \right\rbrace .
$$
Analogously, we define
$$
B_r^{H} (x) := \left\lbrace \xi \in \RR^N : H (x - \xi) < r \right\rbrace .
$$
The sets $B_r^{H_0} (x)$ are named {\it Wulff balls} of $H$ centered at $x$ with radius $r$, and they are homothetic copies of the unitary Wulff ball $B^{H_0}$. 
We call {\it Wulff shapes} of $H$ centered at $x$ with radius $r$ the boundaries $\pa B_r^{H_0} (x)$ of $B_r^{H_0} (x)$, i.e.,
$$
\pa B_r^{H_0} (x) := \left\lbrace y \in \RR^N : H_0(x-y) = r \right\rbrace .
$$
In what follows, to ease notation,
we will omit the exponent $H_0$, that is, the Wulff ball (of $H$) $B_r^{H_0} (x)$ centered at $x$ of radius $r$ will be denoted just by $B_r(x)$.

Furthermore, from the homogeneity property of~$H$ in~\eqref{Hcond3} we have
\begin{equation}\label{eq:proprietacheserveinconti}
	\langle \na_{\xi} H (\xi ), \xi \rangle = H(\xi) \, , \quad \text{for any } \, \xi \in \RR^N ,
\end{equation}
where the left-hand side is taken to be $0$ when $\xi = 0$.

Given a measurable set $E \subset \RR^N$, the anisotropic perimeter of $E$ in $\Si$ is defined by
$$
P_H ( E ; \Si ) : = \sup \left\lbrace \int_{ E } \dv \left( \si \right) \, dx \, : \, \si \in C_c^1 ( \Si ; \RR^N) , \, H_0 \left( \si  \right) \le 1   \right\rbrace .
$$
When $H$ is the Euclidean norm, we recover the usual Euclidean perimeter of $E$ in $\Si$
$$
P ( E ; \Si ) : = \sup \left\lbrace \int_{  E } \dv \left( \si  \right) \, dx \, : \, \si \in C_c^1 ( \Si ; \RR^N) , \,  \left| \si \right| \le 1   \right\rbrace .
$$

Notice that 
\begin{equation}\label{poiuytre}
{\mbox{$P_H ( E ; \Si )$ is finite if and only if $P ( E ; \Si )$ is finite.}}\end{equation}
Indeed, all norms in $\RR^N$ are equivalent. Thus, there exist two positive constants $k_1 \le k_2$ such that $H$ satisfies
$$
k_1 | \xi| \le H(\xi) \le k_2 | \xi| \quad \text{ for any } \, \xi \in \RR^N .
$$
This and the homogeneity of $H$ give that
$$
\frac{1}{k_2} | \xi| \le H_0 (\xi) \le \frac{1}{k_1} | \xi| \quad \text{ for any } \, \xi \in \RR^N ,
$$
and hence that
\begin{equation}\label{eq:relazione equivalenza perimetro anisotropo e isotropo}
k_1 P (E ;\Si) \le P_{H} ( E ;\Si) \le k_2 P ( E ; \Si),
\end{equation}
which gives~\eqref{poiuytre}.

%
%
If $E$ is a set of finite perimeter, then we can write that
$$
P ( E ; \Si ) = \cH^{N-1}( \Si \cap \pa^* E ) 
$$
and that
\begin{equation}\label{eq:def con reduced boundary di anisotropic perimeter}
P_H ( E ; \Si ) = \int_{\Si \cap \pa^* E} H( \nu ) \, d \cH^{N-1} ,
\end{equation}
where $\pa^* E$ is the reduced (or measure theoretic) boundary of $E$ on which an outer (measure theoretical) unit normal vector $\nu$ is defined (see \cite{Giusti} and, for the anisotropic setting, \cite{AB}).

Recalling the definition of the weighted anisotropic perimeter in~\eqref{eq:defweighted perimetro smooth}
and of the weighted volume of a set in~\eqref{WEIVOL},
we notice that, if~$w$ is positively homogeneous of degree~$\lambda\ge0$
and~$D$ is the quantity defined in~\eqref{Ddef},
the following\footnote{We stress
that the right-hand side of~\eqref{eq:relazione volume perimetro}
means ``$D$ multiplied by $w (\Si \cap B)$'' (recall also the notation in~\eqref{WEIVOL});
in particular, no confusion should arise with the derivative of~$w$.}
relation
holds true
\begin{equation}\label{eq:relazione volume perimetro}
P_{w, H} (B ; \Si) =  D\;w (\Si \cap B)  ,
\end{equation}
where $B$ is the unitary Wulff ball,
see~\cite[Formula (1.14)]{CRS}.

For more details on $P_{w,H} ( E ; \Si ) $ and for more general definitions, we refer the reader to \cite{BBF} and \cite{CRS} (see also \cite{CGPRS}).

\section{Additonal ingredients for the proof of the main results}\label{INGREDIE}

We collect here
some other auxiliary results that turn out to be handy for the
proof of the main theorems, such as
a maximum principle in cones, a Stampacchia-type result,
some geometric differential inequalities and a 
Pohozaev-type identity.

We start by proving the following maximum principle in cones.

\begin{lem}[Maximum principle in cones]
	\label{lem:maximum principle in cones}
	Let $u$ be a (weak) solution of
	\begin{equation*}
		\begin{cases}
			-\dv\left\{ w(x) \, H( \na u )^{p-1}\, \na_{ \xi} H( \na u) \right\}= \widetilde{f}(x) \, w(x) \ & \mbox{ in } \ \Si \cap \Om, 
			\\
			u=0 \ & \mbox{ on } \ \Ga_0 ,
			\\
			\langle \na_\xi H(\na u) , \nu \rangle = 0 \ & \mbox{ on } \ \Ga_1 
			\setminus \left\lbrace 0 \right\rbrace ,
		\end{cases}
	\end{equation*}
	where the weight $w$ is continuous in $\ol{\Si}$ and positive and locally Lipschitz in $\Si$.
	If $\widetilde{f} \in L^\infty(\Si \cap \Om)$
	is nonnegative, then $u$ is nonnegative. In particular, if $\widetilde{f} \equiv 0$, then $u \equiv 0$. 
\end{lem}
\begin{proof}
	We set $u^- := \max \left\lbrace 0, -u  \right\rbrace$ and we compute
	\begin{equation*}
		\begin{split}
			0 \le \int_{ \Si \cap \Om }   H^p (\na u^-) \, w(x) \, d \cH^{N} 
			& = \int_{ \left\lbrace u<0 \right\rbrace }  H^p (\na u) \, w(x) \, d \cH^{N}
			\\
			& = \int_{ \left\lbrace u<0 \right\rbrace }   H( \na u )^{p-1}\, \langle \na_{ \xi} H( \na u) , \na u \rangle \, w(x) \, d \cH^{N}
			\\
			& = \int_{ \left\lbrace u<0 \right\rbrace }  u \, \widetilde{f} \, w(x) \,  d \cH^{N} \le 0 ,
		\end{split}
	\end{equation*}
	where in the second equality we used \eqref{eq:proprietacheserveinconti} and in the third equality we used the boundary value problem.
	Thus, $\na u^- = 0$ a.e. in~$\Si \cap \Om$.
	Moreover, by Poincar\'e inequality (which holds true since $u^- = 0$ on $\Ga_0$ and $\cH^{N-1}(\Ga_0)>0$) we find that $u^-=0$, and hence $u \ge 0$ a.e. in $\Si \cap \Om$. 
\end{proof}

In what follows, we will exploit the following result regarding the singular set of solutions to the anisotropic weighted equation considered in the present paper. In the isotropic unweighted setting, the result is due to Lou (see \cite[Theorem 1.1 and Corollary 1.1]{Lo}). We mention that the results in~\cite{Lo} hold true for more general nonlinearities, namely $f \in L^q (\Om)$ with $q>p/N$ and $q\ge 2$. 

\begin{lem}\label{lem:Lou generalization}
	Let $E \subset \RR^N$ be a bounded domain. Let~$f \in L^\infty_{\mathrm{loc}}(\RR)$ and $w \in C(\ol{E})$, with $w$ positive and locally Lipschitz in $E$ .
	Let $u$ be a bounded weak solution of
	\begin{equation}\label{eq:EQUAZIONE CON PESO PER LOU}
		-\dv\left\{ w(x) \, H( \na u )^{p-1}\, \na_{ \xi} H( \na u) \right\}= f(u) \, w(x) \ \mbox{ in } \ E ,
	\end{equation}
	that is, 
	$u \in W^{1,p} (E) \cap L^\infty (E)$ and
	\begin{equation}\label{eq:weak solution}
		\int_E  H( \na u )^{p-1}\, \langle \na_{ \xi} H( \na u) , \na \phi \rangle \, w(x) \, d \cH^{N} = \int_E f \, \phi \, w(x) d \cH^N \quad \text{ for any } \phi \in C_c^{\infty} (E).
	\end{equation}
	Then, $f=0$ a.e. in $\left\lbrace x\in E \, : \, \na u(x) =0 \right\rbrace $.
\end{lem}

We remark that we are not making any
assumptions on the regularity of $E$ in Lemma~\ref{lem:Lou generalization}.
As a matter of fact, in what follows we are going to apply this result with $E:= \Si \cap \Om$.

\begin{proof}[Proof of Lemma~\ref{lem:Lou generalization}]
	We start by proving that
	\begin{equation}\label{eq:second order regularity}
		a(\na u) :=  H( \na u )^{p-1}\, \na_{ \xi} H( \na u) \in W^{1,2}_\mathrm{loc} (E) .
	\end{equation}
		To prove this, we first observe that
		\begin{equation}\label{firstyrwtguy876976}
		{\mbox{if $\widetilde{f} \in L^\infty_{\mathrm{loc}} (E)$, then a solution of
		$
		- \dv \left( a ( \na u ) \right) = \widetilde{f}$ in~$ E$ 
		satisfies \eqref{eq:second order regularity}.}}		\end{equation}
		Indeed, when $H$ is the Euclidean norm, this result can be found in \cite[Lemma 2.1]{Lo}. 
		The anisotropic version of this result (for uniformly elliptic norms $H \in C^2 (\RR^N \setminus \left\lbrace 0 \right\rbrace)$) will appear in~\cite{ACF}.
				
		Now, in order to prove \eqref{eq:second order regularity}, we notice that \eqref{eq:EQUAZIONE CON PESO PER LOU} can be rewritten as
		$$
		-\dv\left\{  \, H( \na u )^{p-1}\, \na_{ \xi} H( \na u) \right\} =  f(u) + H( \na u )^{p-1}\, \left\langle \frac{\na w}{w}, \, \na_{ \xi} H( \na u)  \right\rangle   \ \mbox{ in } \ E .
		$$
		Now we set 
		$$
		\widetilde{f}(x):= f \left(u \left(x \right)\right) + H \left( \na u(x) \right)^{p-1}\, \left\langle \frac{\na w(x)}{w(x)}, \, \na_{ \xi} H \left( \na u \left( x \right) \right) \right\rangle
		$$
		and we observe that, in light of our assumptions on $f$ and $w$, and the interior $C^{1,\ga}$-regularity of $u$ (which holds true by \cite{To}), $ \widetilde{f} \in L^\infty_{\mathrm{loc}} (E) $. 
		Here, we used that $u$ is bounded in order to apply \cite{To}. We also used that $\frac{H(\xi)^p}{p} \in C^1(\RR^N)$ (see e.g., \cite[Lemma 2.2]{CFV}). Thus, \eqref{eq:second order regularity} follows
		from~\eqref{firstyrwtguy876976}.
	
Now, for any $\varphi \in C_c^\infty (E)$ we set
	$$
	\phi(x) := \frac{ \left|  a( \na u(x) ) \right| }{\ve + \left|  a( \na u (x)) \right| } \,  \varphi (x) 
	$$	
	and we notice that $\phi \in W^{1,2}_0 (E)$.
	By using $\phi$ as a test function in \eqref{eq:weak solution}, we obtain that
	\begin{equation}\label{eq:parte111111111}
		\begin{split}
			\int_E  \frac{ \left|  a( \na u) \right| }{\ve + \left|  a( \na u) \right| } \,  \varphi (x)  \, f \,  w(x) \, d \cH^N 
			& =	
			\int_E \frac{ \left|  a( \na u ) \right| }{\ve + \left|  a( \na u ) \right| }  \, \langle  a( \na u ) , \na \varphi \rangle \,  w(x) \, d \cH^{N} 
			\\
			&\qquad+
			\ve \int_E \frac{  \langle  a( \na u ) , \na \left( |  a( \na u ) | \right) \rangle  }{ \left( \ve + \left|  a( \na u ) \right| \right)^2 } \,  \varphi (x)  \,  w(x) \, d \cH^{N}  .
		\end{split}
	\end{equation}
	By the Cauchy-Schwartz inequality we have that 
	$$
	\ve \frac{  \langle  a( \na u ) , \na \left( \left|  a( \na u ) \right| \right) \rangle  }{ \left( \ve + \left|  a( \na u ) \right| \right)^2 }
	\le \left|  \na \left( |   a(\na u) | \right)  \right| ,
	$$
	and the last term is in $L^1 (E)$ (actually in $L^2(E)$) and does not depend on $\ve$. 
	Notice also that
	\begin{equation}\label{eq:parte222222222}
		\int_{ E \setminus \left\lbrace \na u =0 \right\rbrace  } \frac{ \left|  a( \na u ) \right| }{\ve + \left|  a( \na u ) \right| } \,  \varphi (x)  \, f \, d \cH^N 
		=
		\int_E \frac{ \left|  a( \na u ) \right| }{\ve + \left|  a( \na u ) \right| } \,  \varphi (x)  \, f \, d \cH^N .
	\end{equation}
	Thus, by letting $\ve \to 0^+$, \eqref{eq:parte111111111}, \eqref{eq:parte222222222} and Lebesgue's dominated convergence theorem give that
	$$
	\int_{E \setminus \left\lbrace \na u =0 \right\rbrace }   \varphi (x)  \, f \, w(x) d \cH^N 
	=
	\int_{ E   }  \langle   a( \na u ) , \na \varphi \rangle \, w(x)  \, d \cH^{N} ,
	$$
	for any $\varphi \in C_c^\infty(E)$. Hence, this and~\eqref{eq:weak solution} lead to
	$$
	\int_E   \varphi (x)  \, f \,  w(x)  \, d \cH^N = \int_{E \setminus \left\lbrace \na u =0 \right\rbrace }   \varphi (x)  \, f \,  w(x)  \, d \cH^N \quad \text{ for any } \varphi \in C_c^\infty(E),
	$$
	and the desired result follows.
\end{proof}

In what follows, for $t \in \left(- \infty , M \right]$ (where $M$ is that defined in \eqref{eq:def M in Main thm}), we consider the sets
$$u^{-1}(t):= \left\lbrace u=t \right\rbrace:= \left\lbrace x \in \Si \cap \Om: u(x) = t \right\rbrace 
\quad \text{ and } \quad
\left\lbrace u > t \right\rbrace := \left\lbrace x \in \Si \cap \Om : u(x) > t \right\rbrace ,
$$
which are contained in $\Si \cap \Om$ by definition.

Also, we consider the distribution-type functions:
\begin{equation}\label{eq:def weighted I J}
	I(t) := \int_{ \left\lbrace u > t \right\rbrace } f(u) w(x) \, d\cH^{N} , \qquad \mu(t):= w( \left\lbrace u > t \right\rbrace ) ,
\end{equation}
and
\begin{equation}\label{eq:def weighted K}
	K:= I^\al \mu^\be ,
\end{equation}
for any~$\alpha$, $\beta\in\RR$.

In the sequel, we will specify our choice of parameters~$\alpha$ and~$\beta$
as follows
\begin{equation}\label{eq:def weighted albe}
	\al:=p'= \frac{p}{p-1}  \quad {\mbox{ and }}\quad \be:= \frac{p-D}{D(p-1)} ,
\end{equation}
where~$D$ is the quantity defined in~\eqref{Ddef}.

We notice that, when $w \equiv 1$ (and hence 
$D=N$), these distribution-type functions are
precisely those used in \cite{Se} in the isotropic unweighted case (when $ \Si = \RR^N$).

\begin{lem}\label{lem:weighted con4items}
	Let $\al$, $\be$ be arbitrary real numbers and $\Om \subset \RR^N$ be a bounded smooth domain.
	Assume that~$u \in  C^{0,1}(\ol{\Si \cap \Om}) \cap C^1(\Si \cap \Om)$ is nonnegative, $u=0$ on $\Ga_0$, and let $w$ be as in Lemma \ref{lem:maximum principle in cones}.
	Let $f \in L^{\infty}_{\mathrm{loc}}(\left[0, \infty \right))$ be nonnegative
	and let $I$, $\mu$, $K$ be those defined in \eqref{eq:def weighted I J} and \eqref{eq:def weighted K}. Then:
	\begin{enumerate}[(i)]
		\item The functions $I$, $\mu$ and $K$ are a.e. differentiable and
		\begin{equation}\label{eq:weighted derivative K}
			- K'(t)= \left[  \al I(t)^{\al -1} \mu(t)^{\be} f(t) + \be I(t)^\al \mu(t)^{\be -1} \right] \left( - \mu'(t) \right) \quad \text{for a.e. } t.
		\end{equation}
		
		\item For a.e. $t \in (0,M)$, it holds that
		\begin{equation}\label{eq:weighted inequalitymuprimo}
			- \mu ' (t) \ge \int_{ \left\lbrace u=t \right\rbrace } \frac{w}{ | \na u |} \, d \cH^{N-1}.
		\end{equation} 
	\end{enumerate}
	Assume furthermore that $u$ is a weak solution of \eqref{eq:GENERALPB weighted anisotropic problem in cones}. Then:
	\begin{enumerate}[(i)]
	\setcounter{enumi}{2}	
		\item  If hypothesis \eqref{eq:GENERAL condizione b} in Theorem \ref{thm:MAIN GENERAL} holds,
		then~$I$, $\mu$ and $K$ are absolutely continuous functions for $t < M$.
		
		\item Suppose in addition that~$\alpha$ and~$\beta$ are as in~\eqref{eq:def weighted albe}.
		Under the assumptions \eqref{eq:GENERAL condizione a} and \eqref{eq:GENERAL condizione b} in Theorem \ref{thm:MAIN GENERAL}, $K(t)$ defined in \eqref{eq:def weighted K} is nonincreasing for $t \in (0,M)$.
	\end{enumerate}
\end{lem}

\begin{proof}
	We take inspiration from the arguments used in \cite[Lemma 4]{Se} in the isotropic unweighted case~$\Si= \RR^N$. More details and related questions can be found, e.g., in \cite{BZ, Tal, Fe}.
	
	{\emph{(i)}} Since $I$ and $\mu$ are nonincreasing (by definition), they are a.e. differentiable. Also, they define nonpositive Lebesgue-Stieltjes measures $d I$ and $d \mu$ on $(0,M)$. By definition of Lebesgue integral, we find that
$$
I(t) = \int_{ \left\lbrace u > t \right\rbrace } f(u) w(x) \, d\cH^{N} = - \int_t^{M^+} f( \tau ) \, d \mu( \tau ) ,
$$	
and hence $dI = f \, d \mu$.
Thus, $I'(t) = f(t) \mu '(t)$ for $d \mu$-a.e. $t$. Now, since $w>0$ in $\Si$ and $| \na u|$ is bounded in $ \left\lbrace u \ge t \right\rbrace$, $\mu$ is strictly decreasing. As a consequence, we have that the Lebesgue measure on $(0,M)$ is absolutely continuous with respect to $d \mu$. Thus, we get that
$$
I'(t) = f(t) \mu '(t) \quad \text{ for a.e. } t ,
$$
and \eqref{eq:weighted derivative K} easily follows.
In all the paper, unless otherwise indicated, a.e. is always with respect to the Lebesgue measure.
\medskip

	{\emph{(ii)}} We define
	$$
	\mu_0 (t):= w \big( \left\lbrace u>t, \, |\na u| >0 \right\rbrace \big).
	$$
	Let $\ve>0$ and we use coarea formula to compute
	\begin{equation}\label{eq:coarea weighted}
		\begin{split}
			w \big( \left\lbrace u>t, \, |\na u| > \ve \right\rbrace \big) 
			& = 
			\int_{\Si \cap \Om} |\na u| \frac{\chi_{ \left\lbrace u>t, \, |\na u| > \ve \right\rbrace } \, w(x) }{ |\na u|}  \, d\cH^{N}
			\\
			& = \int_t^M \int_{u^{-1}(\tau)} \frac{\chi_{ \left\lbrace u>t, \, |\na u| > \ve \right\rbrace } \, w(x) }{ |\na u|} \, d\cH^{N-1} \, d \tau
			\\
			& = \int_t^M \int_{ u^{-1}(\tau) \cap \left\lbrace |\na u| > \ve \right\rbrace } \frac{w(x)}{ |\na u|} \, d\cH^{N-1} \, d \tau .
		\end{split}
	\end{equation}
	
	By monotone convergence, letting $\ve \to 0$, we find that~\eqref{eq:coarea weighted} still holds true for $\ve = 0$ (and arbitrary $t$). Thus, $\mu_0(t)$ is absolutely continuous and
	\begin{equation}\label{eq:mu primo 0 weighted}
		- \mu_0'(t)= \int_{ u^{-1}( t ) \cap \left\lbrace |\na u| > 0 \right\rbrace } \frac{w(x)}{ |\na u|} \, d\cH^{N-1} \, d \tau , \quad \text{ for a.e. } t \in (0,M).
	\end{equation}
	
	Again by using coarea formula we also find
	$$
	0= \int_{\Si \cap \Om} |\na u| \,  \chi_{ \left\lbrace u>t, \, |\na u| = 0 \right\rbrace } \, d \cH^{N} = \int_t^M \cH^{N-1} \left( u^{-1} ( \tau ) \cap \left\lbrace |\na u| = 0 \right\rbrace \right) \, d \tau ,
	$$
	and hence 
\begin{equation}\label{eq:singular set on level sets has zero N-1-measure}
\cH^{N-1} \left( u^{-1} ( t ) \cap \left\lbrace |\na u| = 0 \right\rbrace \right) = 0 
\quad \text{for a.e. } \, t .
\end{equation}
	
	Thus, we can replace \eqref{eq:mu primo 0 weighted} with
\begin{equation}\label{KMSFBDOMSNAUF}
	- \mu_0'(t)= \int_{ u^{-1}( t ) } \frac{w(x)}{ |\na u|} \, d\cH^{N-1} \, d \tau , \quad \text{ for a.e. } t \in (0,M).
	\end{equation}
Also, for a.e. $t \in (0,M) $ (where both $\mu'(t)$ and $\mu_0' (t)$ exist) we have that
	$$
	- \mu'(t) = \lim_{s \to t^+} \frac{ w \left( s \ge u >t \right)  }{ s - t } \ge \lim_{s \to t^+} \frac{ w \left( s \ge u >t , \, | \na u| > 0 \right)  }{ s - t } = - \mu_0'(t) .
	$$
	Inequality \eqref{eq:weighted inequalitymuprimo} now
	follows from the latter equation and~\eqref{KMSFBDOMSNAUF}. We notice additionally
	that the equality sign holds for a.e. $t$ if $\left\lbrace | \na u | = 0 \right\rbrace$ has zero $\cH^{N}$-measure.
	\medskip
	
		{\emph{(iii)}} If $p<D$ and $\phi \le f \le \frac{D p}{D-p} \phi $ for some
		nonincreasing function~$\phi \ge 0$, then a solution
	of \eqref{eq:GENERALPB weighted anisotropic problem in cones} satisfies 
	$$-\dv\left\{  H( \na u )^{p-1}\, \na_{ \xi} H( \na u) \right\}=0 \quad \text{in } \, \left\lbrace u \ge t_0 \right\rbrace ,$$
	where $t_0 \in \left[ 0, \infty \right]$ satisfies that $\phi (t)>0$ for $t<0$ and $\phi (t) \equiv 0$ for $t >t_0$. Hence, by Lemma \ref{lem:maximum principle in cones}, if $t_0 < +\infty$, we will have that
	$u \equiv t_0 = M$ in $\left\lbrace u \ge t_0 \right\rbrace$. Therefore, for every $t<M$ it holds
	that 
	$$-\dv\left\{  H( \na u )^{p-1}\, \na_{ \xi} H( \na u) \right\}=f(u) \ge \phi(u) \ge \phi(t) >0 \quad \text{in } \, \left\lbrace 0 \le u < t \right\rbrace .$$
	By using Lemma \ref{lem:Lou generalization} (with $E := \Si \cap \Om$), we find that $f(u)$ vanishes a.e. in the set $\left\lbrace | \na u|=0 \right\rbrace \cap \left\lbrace 0 \le u <t \right\rbrace$. 
	Since $f(u) \ge \phi(u) \ge \phi(t) >0$ in $\left\lbrace 0 \le u < t \right\rbrace$, this is only
	possible if the singular set $\left\lbrace | \na u|=0 \right\rbrace \cap \left\lbrace 0 \le u <t \right\rbrace$ has zero measure. Therefore, we have
	$$
	\mu(t)= \mu_0 (t) +  w  \left( \left\lbrace u=M \right\rbrace \right) \quad \text{for every }  \, t <M .
	$$
	%
	%
	%
	%
	Since $\mu_0 (t)$ is absolutely continuous (as noticed in item (ii)), $\mu$ is an absolutely continuous function for $t<M$. Thus, also $I$ and $K$ are absolutely continuous for $t<M$.
	\medskip
	
	{\emph{(iv)}} Under hypothesis \eqref{eq:GENERAL condizione a} of Theorem~\ref{thm:MAIN GENERAL},
	 the claim is obvious because $\al$ and $\be$ in \eqref{eq:def weighted albe} are $\ge 0$. On the other hand, under hypothesis \eqref{eq:GENERAL condizione b} we have that
	item (iii) in the statement of
	Lemma~\ref{lem:weighted con4items}
	applies and hence $K$ is absolutely continuous. Moreover, by item (i) we deduce that 
	\begin{equation}\label{eq:daprovare}
		-K' \ge 0 \quad \text{a.e.}
	\end{equation} 
	and hence $K$ is nonincreasing also in this case.
	To prove \eqref{eq:daprovare}, we notice that in light of item (i), $- K'(t)$ has the same sign as 
	$\al f(t) + \be I(t)/ \mu (t)$	, since
	$I$, $\mu$, $- \mu'$ are nonnegative by definition. Thus, since $\be < 0$, we need $I(t)/ \mu (t) +
	( \al / \be )f(t) \le 0$ a.e. Observing that $I(t)/ \mu(t)$ is the mean of $f(u)$ over the superlevel
	set ${u > t}$, we easily conclude that a sufficient condition for $I/ \mu + ( \al / \be )f \le 0$ is
	that $f(s) \le -( \al / \be ) f(t)$, whenever $s > t$.
	And this is satisfied if $\phi \le f \le -(\al / \be ) \phi$
	for some nonincreasing $\phi \ge 0$. Replacing $\al$, $\be$ by their values in \eqref{eq:def weighted albe} we obtain the condition \eqref{eq:GENERAL condizione b} in Theorem \ref{thm:MAIN GENERAL} since $- \al / \be = p D / (D -p) $.
\end{proof}

\begin{lem}[Pohozaev-type identity]
	\label{lem:Pohozaev General}
	 Let $\Sigma$ be an open convex cone in $\RR^N$ and let $\Omega \subset \RR^N$ be a bounded smooth domain such that $\pa \Gamma_0 = \pa \Gamma_1$.
	%
	%
	Let $u \in C^1 \left( (\Si \cap \Om) \cup \Ga_0 \cup ( \Ga_1 \setminus \left\lbrace 0 \right\rbrace ) \right) \cap W^{1,\infty}(\Si \cap \Om )$ 
	be a solution of \eqref{eq:GENERALPB weighted anisotropic problem in cones}, with $f\in L^{\infty}_{ \mathrm{loc} }([0,\infty))$ nonnegative
	and $w$ satisfying~\eqref{eq:assumptions on w for main theorems} and \eqref{eq:assumptions on w for main theorems3}. Let~$D$ be as in~\eqref{Ddef}.
	
	Then, the following identity holds:
	\begin{equation}\label{eq:weighted anisotropic Pohozaev}
		D \int_{\Si \cap \Om} F(u) \, w(x) \, d \cH^{N} +  \frac{p-D}{p} \int_{ \Si \cap \Om } u \, f(u) \, w(x) \, d \cH^{N} = \frac{1}{p'} \int_{\Ga_0 } H (\na u)^p \, \langle x, \nu \rangle \, w(x) \, d \cH^{N-1} ,
	\end{equation}
where
$$ 
F(s) := \int_0^s f(\tau) \, d\tau .
$$
\end{lem}

\begin{proof}
Due to lack of regularity in our setting, in order to prove \eqref{eq:weighted anisotropic Pohozaev}
we argue by approximation.
We approximate $\Si \cap \Om$ by Lipschitz domains $E_\de$ obtained by chopping off a $\de$-neighborhood of $0$ (in the case $0 \in \Gamma_1$) and a $\de$-tubular neighborhood of $\pa \Ga_0$. 

\medskip

{\emph{Step 1.}} We prove that for any $\eta \in C^{0,1}_c( E_\de )$ we have that
\begin{equation}\label{eq:integral identity su E delta}
- \int_{E_\de}  \cR \, \na  \eta \, d \cH^{N} =
\int_{E_\de} \eta(x) \, w(x) \left\lbrace (N+\la) \, F(u)  + \frac{ p - (N+\la) }{ p } u  \, f(u) \right\rbrace d \cH^N ,
\end{equation}
where
\begin{equation*}
	\begin{split}
\cR : = w(x)  \bigg\lbrace x   \Big(  F(u ) - \frac{ H \left( \na u   \right)^p}{ p } \Big) \bigg.
	& + 
	\langle x , \na u  \rangle \,  H \left( \na u 
	\right)^{p-1}   \na_\xi H(\na u ) 
	\\
	& \bigg. +  \left( \frac{ (N+\la) - p }{ p } \right) \, u  \, H \left( \na u  \right)^{p-1}
  \na _\xi H \left( \na u   \right)  \bigg\rbrace .
	\end{split}
\end{equation*}
For any given $\eta \in C^{0,1}_c( E_\de )$, we
consider a smooth (say $C^{1, \ga}$) open set $A \subset \subset E_\de$ such that $\mathrm{supp} ( \eta ) \subset A$.

Furthermore, by setting
	$$
	\psi(t):= \frac{t^p}{p} , \quad \text{ for } t \ge 0 ,
	$$ we see that
	$u$ satisfies the equation
	$$
	- \dv \left( w(x) \, \psi' \left( H( \na u ) \right) \na_\xi H ( \na u )   \right)  = f(u) \, w(x) \quad \text{ in } \Si \cap \Om.
	$$

		For~$t \ge 0$ and $\ve \in (0,1)$, we let
		$$
		\psi_\ve (t) := \psi \left( \sqrt{\ve^2 +t^2 } \right) - \psi( \ve ). 
		$$
		We define $\Psi(t) := \psi'(t) t$
		and $\Psi_\ve (t) := \psi'_\ve (t) t$. From a standard argument (see for instance \cite[Lemma~4.2]{CFV} or \cite[Lemma~3.2]{BC}) we have that
		$$
		\psi_\ve \to \psi \quad \text{ and }\quad \Psi_\ve \to \Psi \, \text{ uniformly on compact sets of } [ 0, +\infty ). 
		$$	
	
		Let $u_\ve$ be solution of
		\begin{equation*}
			\begin{cases}
				\dv \left( w(x) \, \psi_\ve' \left( H( \na u_\ve) \right) \na_\xi H ( \na u_\ve )   \right)  = f(u) \, w(x)  \quad \text{ in } A
				\\
				u_\ve = u \text{ on } \pa A .
			\end{cases}
		\end{equation*} 
		Standard regularity results give that	
		$u_\ve \in C^{1,\ga} (\ol{ A } ) \cap W^{2,2}_{\mathrm{loc} } ( A )$ and $u_\ve \to u$ in $C^1 (\ol{A})$, as $\ve \to 0$. 
		For example, one can obtain $u_\ve$ by minimizing the functional
		$$
		\int_A \left\lbrace \psi_\ve \left( H (\na v) \right)
         - f(u(x)) \, v \right\rbrace w(x) \, d \cH^N
		$$ 
		among functions $v \in W^{1,p}(A)$ with $v-u \in W^{1,p}_0 (A)$,
		and then exploit arguments similar to those used in \cite[Proposition 4.3]{CFV}.
		
		Notice that, since $A \subset \subset \Si$, we have that $w \in C^{0,1} ( \ol{A} )$ and $w>0$ on $\ol{A}$ and hence the results of regularity (up to the boundary) contained in \cite{Li} can be used on $A$.
		
Now we set
	\begin{equation*}
	\begin{split}
\cR_\ve : = w(x)  \bigg\lbrace x   \Big(  F(u ) - \psi_\ve \left( H \left( \na u_\ve  \right)  \right) \Big) \bigg.
	& + 
	\langle x , \na u_\ve  \rangle \, \psi_\ve' \left( H \left( \na u_\ve 
	\right)  \right)   \na_\xi H(\na u_\ve  ) 
	\\
	& \bigg. +  \left( \frac{ (N+\la) - p }{ p } \right) \, u_\ve  \, \psi_\ve' \left( H \left( \na u_\ve   \right) \right)
  \na _\xi H \left( \na u_\ve   \right)  \bigg\rbrace
	\end{split}
	\end{equation*}
and we exploit
\cite[Proposition 1]{PS} and \cite[Remark at page 685]{PS} with 
\begin{eqnarray*}&&
	\mathcal{F} := \mathcal{F}_\ve (x , \na u_{\ve} ) := \left\lbrace \psi_\ve \left( H (\na u_\ve) \right) - F(u(x)) \right\rbrace w(x) , 
	\qquad h:=x ,
	\qquad a :=  \frac{ (N+\la) - p }{p}\\
&& {\mbox{and}}	\;\qquad \cG (x) := - f(u(x)) w(x) ,
	\end{eqnarray*}
	equation~\eqref{eq:proprietacheserveinconti}, and the fact that $\langle x, \na w \rangle = \la \, w(x)$ in $\Si$
	(which holds true in light of the $\la$-homogeneity of~$w$), to see that
\begin{equation}
\begin{split}
\dv (\cR_\ve) = 
& (N+\la) \, F(u) \, w(x) + \frac{ p - (N+\la) }{ p } \, u_\ve \, f(u) \, w(x)
\\
& \quad+ (N + \la) \, w(x) \left\lbrace \frac{\Psi_\ve ( H( \na u_\ve ))  }{p} - \psi_\ve ( H( \na u_\ve )) \right\rbrace
\\
& \quad+ w(x) \,  \left\lbrace F'(u) \langle x , \na u \rangle - f(u) \langle x , \na u_\ve \rangle  \right\rbrace
\end{split}
	\end{equation}
in the weak sense in $A$.
By recalling that $\eta \in C_c^{0,1} ( E_\de )$, and that $\mathrm{supp} (\eta) \subset A$, we thus can write that
\begin{equation}\label{djeifheugeuo}
\begin{split}
- \int_{E_\de}  \cR_\ve \, \na  \eta \, d \cH^{N} = 
& \int_{E_\de} \eta(x) \, w(x)
\left\lbrace (N+\la) \, F(u)  + \frac{ p - (N+\la) }{ p } u_\de^\ve  \, f(u) 
\right\rbrace d \cH^N 
\\
&\quad + (N + \la) \int_{E_\de} \eta(x)   \, w(x) \left\lbrace \frac{\Psi_\ve ( H( \na u_\ve ))  }{p} - \psi_\ve ( H( \na u_\ve )) \right\rbrace d \cH^N 
\\
& \quad+ \int_{E_\de} \eta(x)\, w(x) \,  \left\lbrace F'(u) \langle x , \na u \rangle - f(u) \langle x , \na u_\ve \rangle  \right\rbrace d \cH^N .
\end{split}
\end{equation}
We also note that
$$
\frac{\Psi ( H( \na u ))  }{p} - \psi ( H( \na u )) = 0.
$$
Moreover, since~$F'(t)=f(t)$ for a.e. $t$, and hence $F'(u(x))=f(u(x))$ for
a.e. $x \in \left\lbrace \na u \neq 0 \right\rbrace$, we thus have that
\begin{eqnarray*}&&
\int_{E_\delta}\eta(x)\,w(x)\, F'(u)  \langle x , \na u \rangle \, d \cH^N  = \int_{ E_\delta\cap
\left\lbrace \na u \neq 0 \right\rbrace}\eta(x)\,w(x)\, f(u)  \langle x , \na u \rangle \, d \cH^N 
\\&&\qquad\qquad= \int_{E_\delta}\eta(x)\,w(x)\, f(u)  \langle x , \na u \rangle \, d \cH^N .
\end{eqnarray*}
This implies that
$$
\int_{E_\de} \eta(x)\, w(x) \, \langle x , \na u \rangle  \left\lbrace F'(u)  - f(u)  \right\rbrace d \cH^N =
0 .
$$
Hence, by taking the limit for $\ve \to 0$ in~\eqref{djeifheugeuo}, we obtain \eqref{eq:integral identity su E delta}.
%
%

\medskip

{\emph{Step 2.}}
For $k \ge 1$ we define $\phi_k : \RR \to \left[ 0 , 1 \right]$ as follows
$$
\phi_k (s) := 
\begin{cases}
0 \quad & \text{ if } \, s \le \frac{1}{k},
\\
k s - 1 \quad & \text{ if } \,  \frac{1}{k} < s < \frac{2}{k},
\\
1 \quad & \text{ if } \, s  \ge  \frac{2}{k} .
\end{cases}
$$
Then we define $\eta_k \in C_c^{0,1}(E_\de)$ as
\begin{equation*}
\eta_k (x) := \phi_k( \mathrm{dist}(x, \RR^N \setminus E_\de)).
\end{equation*}
In this way we have that 
$$\eta_k (x) \to 1 \quad \text{ for every } \, x \in E_\de$$
and that
(see, e.g., \cite[Formula (21)]{DMS}, \cite[Section 7]{CDLP} and \cite[Chapter 3]{AFP})
\begin{equation}\label{eq:messa per sostituire (3.21)}
\lim_k \int_{E_\de} \langle v , \na \eta_k   \rangle d \cH^N = - \int_{\pa E_\de} \langle v , \nu \rangle d \cH^{N-1} , \quad \text{ for any } \, v \in C( \ol{E}_\de ; \RR^N ).
\end{equation}

Now we choose $\eta:=\eta_k$ in \eqref{eq:integral identity su E delta} and we take the limit as $k \to \infty$, using \eqref{eq:messa per sostituire (3.21)} (with $v=\cR$) and taking into account that the boundary conditions in \eqref{eq:GENERALPB weighted anisotropic problem in cones} and the geometry of $\Si$ give that
		$$
		u=0 \quad \text{ and } \quad F(u)=0 \quad \text{ on } \Ga_{0,\de},
		$$
		$$
		\langle \na_\xi H (\na u) , \nu \rangle = 0 \quad \text{ and } \quad \langle x, \nu \rangle = 0 \quad \text{ on } \Ga_{1,\de} ,
		$$
		where we have set 
		$$
		\Ga_{0,\de} := \Ga_0 \cap \pa E_\de , \quad \Ga_{1, \de } := \Ga_1 \cap \pa E_\de \quad \text{ and }\quad
		 \Ga_\de := \pa E_\de \setminus \left( \Ga_{0,\de} \cup \Ga_{1,\de} \right) .
		$$ 
Thus,  we obtain that
		\begin{equation*}
			\begin{split}
        \int_{\Ga_{0, \de} } w(x) & \left\lbrace \langle x, \na u \rangle \,  H^{p-1} \left( \na u  \right) \langle \na_\xi H(\na u  ) , \nu \rangle - \langle x, \nu \rangle	 \frac{ H \left( \na u  \right)^p }{ p } \right\rbrace d \cH^{N-1} 
         + \int_{\Ga_\de} \langle \cR , \nu \rangle \, d \cH^{N-1}
        \\ &
        =
        \int_{E_\de} w(x) \left\lbrace (N+\la) \, F(u)  - \frac{ (N+\la) - p }{ p } u  \, f(u) \right\rbrace d \cH^N . 
        \end{split}
	    \end{equation*}
        By using that 
        $$
        \na u = - |\na u| \, \nu \quad \text{ on } \, \Ga_{0, \de} ,
        $$ 
        (which holds true in light of the boundary condition on $\Ga_0$ in \eqref{eq:GENERALPB weighted anisotropic problem in cones}) and \eqref{eq:proprietacheserveinconti}, the previous identity becomes
    \begin{eqnarray*}&&
         \frac{1}{ p' } \int_{\Ga_{0, \de} } w(x)   \,
         H \left( \na u  \right)^p \, \langle x, \nu \rangle \, d \cH^{N-1} + \int_{\Ga_\de} \langle \cR , \nu \rangle \, d \cH^{N-1}
       \\& =&
        \int_{E_\de} w(x) \left\lbrace (N+\la)  F(u)  - \frac{ (N+\la) - p }{ p } u  \, f(u) \right\rbrace d \cH^N  ,
        \end{eqnarray*}
        where $p'= p/ (p-1) $.
		Then we take the limit as $\de \to 0$. Since $u \in W^{1, \infty} (\Si \cap \Om)$ and $\cH^{N-1} (\Ga_\de)$ tends
		to~$0$ as~$\de \to 0$, we have that the integral over $\Ga_\de$ tends to $0$. Recalling that $D= N+\la$, \eqref{eq:weighted anisotropic Pohozaev} follows.
\end{proof}

\section{Isoperimetric inequalities in convex cones}\label{sec:isoperimetric inequalities}

In this section, we revisit the existing literature
on isoperimetric inequalities in convex cones
and, providing several technical improvements
to the previous results,
we frame the isoperimetric theory into a convenient setting
for our applications.
%
%
%
\begin{thm}[Weighted anisotropic isoperimetric inequality in cones, \cite{CRS}]
\label{thm:isoperimetric anisotropic weighted cones}
	Let $\Si$ be an open convex cone in $\RR^N$ with vertex at the origin and 
	let $B$ be the unitary Wulff ball centered at the origin. Let $w$ satisfy~\eqref{eq:assumptions on w for main theorems}, \eqref{eq:assumptions on w for main theorems2}, and \eqref{eq:assumptions on w for main theorems3}. Then, for each measurable set $E \subset \RR^N$ with $w (\Si \cap E) < \infty$,
	\begin{equation}\label{eq:weighted anisotropic isoperimetric in cones}
		\frac{P_{w,H} (E ; \Si)}{w( \Si \cap E)^{\frac{D-1}{D}}} \ge \frac{P_{w,H} (B ; \Si)}{w( \Si \cap B)^{\frac{D-1}{D}}} ,
	\end{equation}
	where $D= N + \la$, and $\la$ is that in \eqref{eq:assumptions on w for main theorems}.
\end{thm}

Theorem~\ref{thm:isoperimetric anisotropic weighted cones} has been obtained in by Cabr\'{e}--Ros-Oton--Serra in \cite{CRS} in the more general setting in which $H$ is a gauge. Since we do not need such a generality here, we have stated it assuming
that~$H$ is a norm. 

As mentioned in the introduction, Wulff balls centered at the origin intersected with $\Si$ are always minimizers of \eqref{eq:weighted anisotropic isoperimetric in cones}. Let us discuss about the uniqueness of those minimizers, or, in other words, about the characterization of the equality sign in \eqref{eq:weighted anisotropic isoperimetric in cones}.

In \cite{CRS} the authors do not provide the characterization of the equality case, stating that such a characterization will be postponed in a future work. To the authors' knowledge such a characterization among general sets in the unweighted anisotropic setting of Theorem \ref{thm:isoperimetric anisotropic weighted cones} is still not available in the literature.

We recall that the isoperimetric inequality \eqref{eq:weighted anisotropic isoperimetric in cones} in the unweighted isotropic setting (i.e., when $w \equiv 1$ and $H$ is the Euclidean norm) is due to Lions--Pacella~\cite{LP}. In this setting, they also obtained the characterization of the equality case when $\Si \setminus \left\lbrace 0 \right\rbrace $ is assumed to be smooth.
%
%

Quantitative versions of inequality \eqref{eq:weighted anisotropic isoperimetric in cones} in the unweighted isotropic setting (i.e., when $w \equiv 1$ and $H$ is the Euclidean norm) including the characterization of the equality case (in cones that are not necessarily smooth outside the origin) have been obtained by Figalli-Indrei in \cite{FI}. In Theorem \ref{thm:characterization unweighted anisotropic} below we show that their proof can be extended to the case of a general norm $H$. 

A quantitative version of inequality \eqref{eq:weighted anisotropic isoperimetric in cones} in the weighted isotropic setting (i.e., with $w$ as in Theorem~\ref{thm:isoperimetric anisotropic weighted cones} and $H$ the Euclidean norm) including the characterization of the equality case (in cones that are not necessarily smooth outside the origin) has been recently obtained by Cinti--Glaudo--Pratelli--Ros-Oton--Serra in \cite{CGPRS}.
 
A different proof of Theorem \ref{thm:isoperimetric anisotropic weighted cones} has been obtained by Milman-Rotem in \cite{MR}. There, the authors also provide a characterization under the additional assumption that $\Si \cap E$ is convex\footnote{\cite{MR} has been published in [Adv. Math.262, 2014, 867–908]. 
%
%
In the published version, the convexity assumption in the characterization of minimizers was missing. This has been later fixed in \cite{MR} (see footnote 1 in \cite{MR}).}. Due to the convexity restriction, we cannot apply their characterization in the present paper. 
In fact, in order to obtain Theorms \ref{thm:cones}, \ref{thm:MAIN unweighted in Wulff ball}, \ref{thm:weighted cones} from Theorem \ref{thm:MAIN GENERAL}, we need to use such a characterization on (almost all) the superlevel sets $ \left\lbrace u > t \right\rbrace$, and we do not know a priori whether those sets are convex or not.

\begin{thm}[Characterization of minimizers in the unweighted anisotropic isoperimetric inequality in convex cones]
\label{thm:characterization unweighted anisotropic}
	Let $\Si$ be an open convex cone in $\RR^N$ with vertex at $0$ and let $B$ be the unitary Wulff ball centered at $0$. 
	Then,
	\begin{equation}\label{eq:anisotropic isoperimetric in cones}
		\frac{P_H (E ; \Si)}{\cH^{N}( \Si \cap E)^{\frac{N-1}{N}}} \ge \frac{P_H (B ; \Si)}{\cH^{N}( \Si \cap B)^{\frac{N-1}{N}}} ,
	\end{equation}
	for every measurable set $E \subset \RR^N$ with $\cH^N ( \Si \cap E ) <\infty$.
	
		Moreover, up to rotations, we
		write $\Si= \RR^k \times \tilde{\Si}$, where $0 \le k \le N$ and $\tilde{\Si} \subset \RR^{N-k}$
		is an open convex cone containing no lines, and
the equality sign holds in~\eqref{eq:anisotropic isoperimetric in cones}
if and only if $E$ is a Wulff ball of some radius $r>0$ centered
at $x_0 \in \RR^k \times \left\lbrace 0_{\RR^{N-k}} \right\rbrace $.
\end{thm}

As already mentioned, when $H$ is the Euclidean norm, the proof of Theorem~\ref{thm:characterization unweighted anisotropic}
can be found in \cite{FI}. There, the authors observe that the (isotropic) isoperimetric inequality in cones (namely formula~\eqref{eq:anisotropic-isoperimetric} with $H$ being the Euclidean norm)
can be obtained as an immediate corollary of the anisotropic isoperimetric inequality
\begin{equation}\label{eq:Isoperimetric Anisotropic any convex}
\frac{P_K (E)}{ \cH^N ( E )^\frac{N-1}{N}} \ge N \cH^N (K)^{\frac{1}{N}} .
\end{equation}
Here, $K$ is an open bounded convex set and, for a set $E$ of finite perimeter (i.e., $P ( E ; \RR^N ) < \infty$),
$$
P_K (E):= \int_{\pa^* E} \nr \nu_E (x) \nr_{K_0} \, d \cH^{N-1} ,
$$
where $\pa^* E$ is the reduced boundary of $E$, $\nu_E$ is the measure theoretic outer unit normal, and
$$
\nr \nu \nr_{K_0} := \sup \left\lbrace \langle \nu , z \rangle \, : \, z \in K \right\rbrace .
$$
Inequality \eqref{eq:Isoperimetric Anisotropic any convex} and its characterization of the equality case are well known in the literature. We refer to \cite{Di, Scha, DP, Ta3, FM, BM, MS, FI} and references therein.

As noticed in \cite{FI}, the choice $K:=\Si \cap B$, being $B$ the unitary Euclidean ball centered at the origin allows to recover \eqref{eq:anisotropic isoperimetric in cones} when $H$ is the Euclidean norm. This also allows to obtain the desired characterization.
We will show that their argument still applies when $H$ is any norm in $\RR^N$. 

\begin{proof}[Proof of Theorem~\ref{thm:characterization unweighted anisotropic}]
The strategy is to extend to a more general $H$ the arguments used in \cite[Theorem~2.2]{FI}.
To this aim, we set 
$$
K := \Si \cap B , \quad \text{ where } B \text{ is the unitary Wulff ball.}
$$
Let now $E$ be a set contained in $\Si$ and with finite perimeter.
We notice that, if $z \in K$, then $H_0(z) \le 1$, and hence by \eqref{eq:Anisotropic_CS_dual},
$$
\nr \nu_E \nr_{K_0} := \sup \left\lbrace \langle \nu_E , z \rangle \, : \, z \in K \right\rbrace \le H( \nu_E  ) .
$$
By definition of $\nr \cdot \nr_{K_0}$ it easily follows that $\nr \nu_\Si \nr_{K_0} = 0$ for $\cH^{N-1}$-a.e. $x \in \pa \Si \setminus \left\lbrace 0 \right\rbrace$, and hence also for $\cH^{N-1}$-a.e. $x \in \pa \Si \cap \pa^* E $.
Thus, 
\begin{equation*}
\begin{split}
P_K (E)
& = \int_{\pa^* E} \nr \nu_E (x) \nr_{K_0} \, d \cH^{N-1}
\\
& = \int_{\Si \cap \pa^* E} \nr \nu_E (x) \nr_{K_0} \, d \cH^{N-1}
\\
& \le \int_{\Si \cap \pa^* E} H( \nu_E (x) ) \, d \cH^{N-1}
\\
& = P_H (E; \Si) ,
\end{split}
\end{equation*}
where in the last identity we used \eqref{eq:def con reduced boundary di anisotropic perimeter}.
In light of the inequality
\begin{equation}\label{eq:confronto perimetri K H}
P_K (E) \le P_H (E; \Si) 
\end{equation}
and recalling \eqref{eq:relazione volume perimetro} (with $w \equiv 1$), it is clear that \eqref{eq:anisotropic isoperimetric in cones} follows from \eqref{eq:Isoperimetric Anisotropic any convex}.

Now assume that $E$ satisfies the equality in \eqref{eq:anisotropic isoperimetric in cones}.
In light of \eqref{eq:confronto perimetri K H}, $E$ will also satisfy the equality in \eqref{eq:Isoperimetric Anisotropic any convex}. 
By rescaling, if necessary, we may assume $\cH^N (E)=\cH^N (K)$. 
Thus, we have that
\begin{equation}\label{eq:charisop step 0}
P (K ; \Si ) = N \cH^N (K) = P_K (E) = P_H (E; \Si) ,
\end{equation}
where the first equality follows from \eqref{eq:relazione volume perimetro} (with $w \equiv 1$).

By \cite{FM} (see also~\cite[Theorem A.1]{FMP}), we obtain that 
\begin{equation}\label{eq:charisop step 1}
E = K + a \quad \text{ with } \, a \in \ol{ \Si} .
\end{equation}

Notice that, for any $v \in \ol{ \Si }$ it holds that
\begin{equation}\label{eq:charisop step 2}
P_H (v+K \, ; \Si)= P_H ( K ; \Si) + P_H (S_v) ,
\end{equation}
where
$$
S_v := \left\lbrace x \in \pa^* \Si \cap B \, : \, \langle \nu_\Si(x) , v \rangle \neq 0 \right\rbrace
= \left\lbrace x \in \pa^* \Si \cap B \, : \, \langle \nu_\Si(x) , v \rangle < 0 \right\rbrace .
$$

By putting together \eqref{eq:charisop step 0}, \eqref{eq:charisop step 1} and~\eqref{eq:charisop step 2}, we get that
$$
P (K ; \Si ) = P_H (E; \Si) = P_H ( K ; \Si) + P_H (S_a) ,
$$
and hence $P_H (S_a)=0$. By recalling \eqref{eq:relazione equivalenza perimetro anisotropo e isotropo} we find that $\cH^{N-1} (S_a)=0 $, and hence that $\langle \nu_\Si(x) , a \rangle = 0$ for $\cH^{N-1}$-a.e. $x \in \pa^* \Si$. As in \cite{FI}, this gives that the distributional derivative of $\chi_\Si$ is zero in the direction defined by $a$ and hence that
$$
\chi_\Si (x) = \chi_\Si (x + t a) \quad \text{ for all } t \in \RR .
$$
Since $\Si= \RR^k \times \tilde{\Si}$, where $0 \le k \le n$ and $\tilde{\Si} \subset \RR^{N-k}$ contains no lines, we find that $a \in \RR^k \times \left\lbrace 0_{\RR^{N-k}} \right\rbrace$. This concludes the proof.
\end{proof}

In the weighted isotropic setting, \cite[Proposition 1.2]{CGPRS} provides the characterization of the equality case.

\begin{thm}[Characterization of minimizers in the weighted isotropic isoperimetric inequality in convex cones, \cite{CGPRS}]
\label{thm:characterization isotropic weighted ISOperimetric}
		Up to rotations, we can write $\Si= \RR^k \times \tilde{\Si}$, where $0 \le k \le N$ and $\tilde{\Si} \subset \RR^{N-k}$ is an open convex cone containing no lines. 
		Assume that~$H$ is the Euclidean norm.
		Then, the equality sign holds
		in~\eqref{eq:weighted anisotropic isoperimetric in cones} (with~$H$ the Euclidean norm)
		 if and only if $E$ is a (Euclidean) ball of some radius~$r>0$ centered at $x_0 \in \RR^k \times \left\lbrace 0_{\RR^{N-k}} \right\rbrace $.
\end{thm}
We point out that \cite[Proposition 1.2]{CGPRS} is stated for $0 \le k < N$. However, notice that in the case $k=N$ (i.e., $\Si = \RR^N $), \eqref{eq:weighted anisotropic isoperimetric in cones} reduces to the unweighted anisotropic isoperimetric inequality in $\RR^N$, whose characterization is well known in the literature (and of course is contained in Theorem \ref{thm:characterization unweighted anisotropic}). In fact, the assumptions on the weight $w$ in Theorem \ref{thm:isoperimetric anisotropic weighted cones} 
force $w$ to be (a positive) constant whenever $\Si = \RR^N$ (see, e.g., \cite[Remark 3.8 and Lemma 3.9]{CR} and \cite{CRS}).

For other results and more information on weighted isoperimetric inequalities, we refer to Frank Morgan's blog \cite{Mo}, \cite{Ch, CRS, CGPRS}, and references therein.

\begin{rem}
\label{rem:decomposizione coni per smooth}
{\rm
In order to describe the results concerning the characterization of minimizers we used that, for any open convex cone $\Si \subseteq \RR^N$, without loss of generality, up to rotations, we can write $\Si= \RR^k \times \tilde{\Si}$, where $0 \le k \le N$ and $\tilde{\Si} \subset \RR^{N-k}$ is an open convex cone containing no lines.
}
\end{rem}

%
%

\section{Proof of Theorems~\ref{thm:MAIN GENERAL},
\ref{thm:cones}, \ref{thm:MAIN unweighted in Wulff ball} and \ref{thm:weighted cones}}\label{FIDIMO}

In this section, we exploit the results
of the previous sections to prove Theorems~\ref{thm:MAIN GENERAL},
\ref{thm:cones}, \ref{thm:MAIN unweighted in Wulff ball} and \ref{thm:weighted cones}.
To this end, the following Gauss-Green-type identity and H\"older-isopertimetric-type inequality will be useful.

\begin{lem}
Suppose that the assumptions of Theorem \ref{thm:MAIN GENERAL} are satisfied.
Let~$I$ be the function defined in~\eqref{eq:def weighted I J}.
Then, we have the
	following:
	\begin{enumerate}[(i)]
		\item The Gauss-Green-type identity holds true, namely
		\begin{equation}\label{eq:Gauss-Green}
			I(t) = \int_{  \left\lbrace u=t \right\rbrace } H(\na u)^{p-1} H( \nu)  \, w(x) \, d\cH^{N-1} \quad \text{for a.e. } \, t \in \left( 0,M \right) .
		\end{equation}
		\item The H\"older-isopertimetric-type inequality holds true, namely
		\begin{equation}\label{eq:Holder-isoperimetric}
			I(t)^{\frac{1}{p}} \, \left( \int_{  \left\lbrace u=t \right\rbrace } \frac{ w(x) }{| \na u |} \, d \cH^{N-1} \right)^{ \frac{p-1}{p} } \ge c \, \mu(t)^{\frac{D-1}{D}}  , \quad \text{for a.e. } t \in \left( 0 , M \right) ,
		\end{equation}
	where 
		\begin{equation}\label{eq:optimal constant anisotropic isoperimetric}
			c: = \frac{ P_{w, H } \left( B ; \Si \right)}{  w( \Si \cap B)^{\frac{D - 1 }{D}}} 
		\end{equation}
		is the optimal constant in the weighted anisotropic isoperimetric inequality in cones \eqref{eq:weighted anisotropic isoperimetric in cones}, being $B$ the unitary Wulff ball centered at the origin. 
		
		Moreover, the equality sign holds in \eqref{eq:Holder-isoperimetric} if and only if $\left\lbrace  u > t \right\rbrace $ is a minimizer of \eqref{eq:weighted anisotropic isoperimetric in cones} and $H( \na u )$ is constant on $\left\lbrace u=t \right\rbrace$.
	\end{enumerate}
\end{lem}

\begin{proof}
	{\emph{(i)}} Since the function $u$ is of bounded variation in $\Si \cap \Om$, the
	coarea theorem for BV functions (see e.g.~\cite[Theorem~1, Section~5.5]{EG}) guarantees that the sets
	$\left\lbrace u > t \right\rbrace$ have finite perimeter for a.e. $t$. For the measure theoretic boundary $\pa^* \left\lbrace u > t \right\rbrace$, 
we see that
	\begin{equation}\label{eq:questaservexxxxxxx}
		\left\lbrace u=t \right\rbrace \cap \left\lbrace | \na u| > 0 \right\rbrace \subset
		\Si \cap  \pa^* \left( \left\lbrace u > t \right\rbrace  \right) \subset  \left\lbrace u = t \right\rbrace .
	\end{equation}
	
	By recalling \eqref{eq:singular set on level sets has zero N-1-measure},
	we conclude that
	\begin{equation*}
		\cH^{N-1} \left(  \left\lbrace u=t \right\rbrace \cap \left\lbrace | \na u| > 0 \right\rbrace \right) = \cH^{N-1} \left( \Si \cap \pa^* \left( \left\lbrace u > t \right\rbrace   \right) \right) \quad \text{for a.e. } t, 
	\end{equation*}
	and the exterior measure theoretical normal vector $\nu$ of $\pa^* \left\lbrace u > t \right\rbrace$ satisfies, a.e. in the sense of the measure $\cH^{N-1}$:
	\begin{equation}\label{eq:normale_theoretic}
		| \na u | \, \nu  = - \na u  \quad \quad \quad
		 \cH^{N-1}-\text{a.e. on } \Si \cap  \pa^* \left( \left\lbrace u > t \right\rbrace  \right).
	\end{equation}
	
	Accordingly, by using \eqref{eq:GENERALPB weighted anisotropic problem in cones} (we are using the differential equation together with the homogeneous boundary condition on $\Ga_1$) and De Giorgi's structure theorem (see \cite[Theorem 15.9]{Maggi} or \cite{Giusti}), we
	have, up to a standard approximation argument,
		\begin{equation*}
		I(t)= \int_{\left\lbrace u > t \right\rbrace} f(u) \, w(x) \, d \cH^N =
		\int_{\Si \cap \pa^* \left\lbrace u > t \right\rbrace} H(\na u)^{p-1} \langle \na_\xi H(\na u), \nu \rangle \, w(x) \, d\cH^{N-1} , \, \, \text{ for a.e. } t \in \left( 0 , M \right)
	\end{equation*}
	and hence, by using \eqref{eq:proprietacheserveinconti} and \eqref{eq:normale_theoretic},
	\begin{equation}
		I(t)=\int_{ \Si \cap \pa^* \left\lbrace u > t \right\rbrace} H(\na u)^{p-1} H \left( \nu \right) \, w(x) \, d\cH^{N-1} , \quad \text{for a.e. } t \in \left( 0 , M \right) .
	\end{equation}
	By \eqref{eq:singular set on level sets has zero N-1-measure} and \eqref{eq:questaservexxxxxxx}, we have that
	\begin{equation*}
		\int_{ \Si \cap \pa^* \left\lbrace u > t \right\rbrace} H(\na u)^{p-1} H \left( \nu \right) \, w(x) \, d\cH^{N-1} =
		\int_{  \left\lbrace u = t \right\rbrace} H(\na u)^{p-1} H \left( \nu \right) \, w(x) \, d\cH^{N-1} , \, \text{ for a.e. } t \in \left( 0 , M\right) ,
	\end{equation*}
	and the conclusion follows.
	\medskip
	
	{\emph{(ii)}} 
		By the isoperimetric inequality \eqref{eq:weighted anisotropic isoperimetric in cones}, we have
	\begin{equation}\label{eq:anisotropic-isoperimetric}
		P_{w, H} \left( \left\lbrace u > t  \right\rbrace  ; \Si \right) \ge c \, \, w \left( \left\lbrace u > t \right\rbrace \right)^{\frac{D-1}{D}} = c \, \mu(t)^{\frac{D-1}{D}} , \quad \text{for a.e. } t \in \left( 0 , M \right) ,
	\end{equation}
	where $c$ is the optimal (weighted anisotropic) isoperimetric constant defined in \eqref{eq:optimal constant anisotropic isoperimetric}. 
	
	Now, by the homogeneity of $H$ and \eqref{eq:normale_theoretic},
$$H(\nu) H(\na u)^{p-1} = H(\nu)^p | \na u |^{p-1}.$$
		By using this and~\eqref{eq:Gauss-Green}, we compute that
	\begin{equation}\label{eq:contoHolder}
		\begin{split}
			I(t)^{\frac{1}{p} } \left( \int_{ \left\lbrace u = t \right\rbrace } \frac{w(x)}{ | \na u | } \, d\cH^{N-1} \right)^{\frac{p-1}{p}} 
			& =
			\left( \int_{ \left\lbrace u = t \right\rbrace }  H(\nu)^p | \na u |^{p-1} \, w(x) \, d\cH^{N-1} \right)^{ \frac{1}{p} } \left( \int_{ \left\lbrace u = t \right\rbrace } \frac{w(x)}{ | \na u | } \, d\cH^{N-1} \right)^{\frac{p-1}{p}}
			\\
			& \ge \int_{ \left\lbrace u=t \right\rbrace } H(\nu) w(x) d \cH^{N-1} 
			\\
			& = P_{w,H} \left( \left\lbrace u > t \right\rbrace ; \Si \right)  
			\\
			& \ge c \, \mu(t)^{\frac{D-1}{D}}.
		\end{split}
	\end{equation}
	Here, the first inequality 
	is a consequence of H\"older's inequality, the second equality follows from \eqref{eq:defweighted perimetro smooth} (by recalling \eqref{eq:singular set on level sets has zero N-1-measure} and \eqref{eq:questaservexxxxxxx}), and the last inequality 
	comes from \eqref{eq:anisotropic-isoperimetric}.
	Thus, \eqref{eq:Holder-isoperimetric} is proved. The characterization of the equality sign in \eqref{eq:Holder-isoperimetric} follows by noting that the equality holds in \eqref{eq:Holder-isoperimetric} if and only if the equality sign holds in both the inequalities in \eqref{eq:contoHolder}. 
\end{proof}

We are now ready for the 

\begin{proof}[Proof of Theorem \ref{thm:MAIN GENERAL}]
We consider the function~$K$ defined in~\eqref{eq:def weighted K}, with~$\alpha$ and~$\beta$ as
in~\eqref{eq:def weighted albe}.
	Since~$K(t)$ is nonnegative and, by item (iv) of Lemma \ref{lem:weighted con4items}, nonincreasing, we have that
	\begin{equation}
		K(0^-) \ge K(0^+) - K(M^-) \ge \int_{0}^M - K'(t) dt .
	\end{equation}
	Combining this inequality with \eqref{eq:weighted derivative K} leads to
	\begin{equation*}
		K(0^-) \ge \int_{0}^M \left[  \al I(t)^{\al -1} \mu(t)^{\be} f(t) + \be I(t)^\al \mu(t)^{\be -1} \right] \left( - \mu'(t) \right) \, dt \quad \text{for a.e. } t.
	\end{equation*} 
	Notice that equality here above holds true when $K$ is absolutely continuous. 
	Being the integrand in the right-hand side (that is $-K'(t)$) and $- \mu'(t)$ nonnegative (and hence so is also the factor in the square bracket), we can use \eqref{eq:weighted inequalitymuprimo} to get
	\begin{equation*}
		K(0^-) \ge \int_{0}^M \left[  \al I(t)^{\al -1} \mu(t)^{\be} f(t) + \be I(t)^\al \mu(t)^{\be -1} \right] \left( \int_{ \left\lbrace u=t \right\rbrace } \frac{ w(x) }{| \na u|}  d \cH^{N-1} \right) \, dt .
	\end{equation*}
	%
	%
	Then we compute
	\begin{equation}\label{eq:contone}
		\begin{split}
			K(0^-) & \ge 
			\int_{0}^M \left[  \al I(t)^{\al -1} \mu(t)^{\be} f(t) + \be I(t)^\al \mu(t)^{\be -1} \right] \left( \int_{ \left\lbrace u=t \right\rbrace } \frac{ w(x) }{| \na u|}  d \cH^{N-1} \right) \, dt 
			\\
			& = \int_{0}^M \left[  \al I(t)^{\al -1 - \frac{1}{p-1} } \mu(t)^{\be} f(t) + \be I(t)^{\al - \frac{1}{p-1} } \mu(t)^{\be -1} \right] I(t)^{ \frac{1}{p-1} }  \left( \int_{ \left\lbrace u=t \right\rbrace } \frac{ w(x) }{| \na u|}  d \cH^{N-1} \right) \, dt
			\\
			& \ge \int_{0}^M c^{ \frac{p}{p-1} } \left[  \al I(t)^{\al -1 - \frac{1}{p-1} } \mu(t)^{\be} f(t) + \be I(t)^{\al - \frac{1}{p-1} } \mu(t)^{\be -1} \right] \mu(t)^{ \frac{p ( D -1)}{(p-1) D } } \, dt ,
		\end{split}
	\end{equation}
	where the last inequality follows from \eqref{eq:Holder-isoperimetric} and the fact that the factor in the square brackets is nonnegative, being $-K' \ge 0 $.
	
	Since \eqref{eq:contone} has been obtained by applying \eqref{eq:Holder-isoperimetric} on almost all the level sets, a necessary condition to get the equality in \eqref{eq:contone} is that, for a.e. $t \in (0, M)$, $\left\lbrace u > t \right\rbrace$ is a minimizer of \eqref{eq:weighted anisotropic isoperimetric in cones} and $H( \na u )$ is constant on $\left\lbrace u = t \right\rbrace$.
	
	Now we notice that the values of $\al$ and $\be$ in \eqref{eq:def weighted albe} are set in order to satisfy
	$$
	\al - 1 -\frac{1}{p-1} = 0 \quad \text{and} \quad \be -1 + \frac{p( D -1)}{(p-1) D } =0 .
	$$
	Thus, by setting 
	$$ 
	F(s):= \int_0^s f(\tau) \, d\tau 
	$$
	and~$p' = p/(p-1)$, \eqref{eq:contone} leads to
	\begin{equation}\label{eq:contone2}
		\begin{split}
			K(0^-) & \ge  \int_{0}^M c^{ p' } \left[  p'  f(t) \, w \left( \left\lbrace u>t \right\rbrace \right) + \frac{p- D }{ D (p-1)} \int_{ \left\lbrace u>t \right\rbrace } f(u) \, w(x) \, d \cH^{N} \right]  \, dt  
			\\
			& = c^{ p' } \int_{0}^M \int_{\Si \cap \Om} \chi_{ \left\lbrace u>t \right\rbrace } w(x) \left(  p'  f(t) + \frac{p- D }{ D (p-1)} f(u)  \right) \, d \cH^{N} \, dt 
			\\
			& = c^{ p' } \left[ p' \int_{\Si \cap \Om } F(u) \, w(x) \, d \cH^{N} +  \frac{p-D}{D(p-1)} \int_{ \Si \cap \Om } u \, f(u) \, w(x) \, d \cH^{N} \right] 
			\\
			& = c^{ p' } \frac{p'}{D} \left[ D \int_{\Si \cap \Om } F(u) \, w(x) \, d \cH^{N} +  \frac{p- D }{ p } \int_{\Si \cap \Om } u \, f(u) \, w(x) \, d \cH^{N} \right] .
		\end{split}
	\end{equation}
	By noting that 
	\begin{equation*}
		K(0^-) = \lim_{t \to 0^-} K(t) =  w \left( \Si \cap \Om \right)^{ \frac{p- D }{D (p-1)} }  \left[ \int_{\Si \cap \Om } f(u) \, w(x) \, d\cH^N \right]^{p'} ,
	\end{equation*}
	we deduce from~\eqref{eq:contone2} that
	\begin{equation}\label{eq:disuguaglianza da sinistra}
		\frac{D}{p' c^{ p' }}  w \left( \Si \cap \Om \right)^{ \frac{ p - D }{ D (p-1)} } \left[ \int_{\Si \cap \Om } f(u) \, w(x) \, d\cH^N \right]^{p'}
		\ge 
		D \int_{ \Si \cap \Om } F(u) \, w(x) \, d \cH^{N} +  \frac{p-D}{p} \int_{ \Si \cap \Om } u \, f(u) \, w(x) \, d \cH^{N} . 
	\end{equation}
	
	Now we recall the weighted anisotropic Pohozaev identity \eqref{eq:weighted anisotropic Pohozaev}:
	\begin{equation*}
		D \int_{\Si \cap \Om } F(u) \, w(x) d \cH^{N} +  \frac{p-D}{p} \int_{\Si \cap \Om } u \, f(u) \, w(x) \, d \cH^{N} = \frac{1}{p'} \int_{\Ga_0} H^p (\na u) \, \langle x, \nu \rangle \, w(x) \, d \cH^{N-1} ,
	\end{equation*}
	where $\Ga_0= \Si \cap \pa \Om$.
	By using for the first time that $\Om$ is a Wulff ball, \eqref{eq:norma H def dual} gives that 
	\begin{equation*} 
		\langle x, \nu \rangle = R \, H(\nu)  \quad \text{ for } x \in \pa \Om \supseteq \Ga_0 ,
	\end{equation*}
	where $R$ is the radius of the Wulff ball $\Om$, and hence
	\begin{equation}\label{eq:usiamo omega ball}
		\frac{1}{p'} \int_{ \Ga_0 } H^p (\na u) \, \langle x, \nu \rangle \, w(x) \, d \cH^{N-1} = \frac{ R }{ p' } \, \int_{ \Ga_0 }  H^p (\na u) \, H(\nu) \, w(x) \, d \cH^{N-1} .
	\end{equation}
	
	By using H\"older's inequality and \eqref{eq:defweighted perimetro smooth} we obtain
		\begin{equation*}
		\begin{split}
			\int_{ \Ga_0}  H^{p-1}(\na u) \, H(\nu) \, w(x) d \cH^{N-1} 
			& = \int_{ \Ga_0} \left[ H(\nu) \, w(x) \right]^{\frac{1}{p} } \, \left[  H^{p-1}(\na u)    \, \left( H(\nu) \,  w(x)\right)^{\frac{1}{p'} } \right]  d \cH^{N-1} 
			\\
			& \le P_{w, H} ( \Om; \Si)^{\frac{1}{p} } \,  \left[ \int_{ \Ga_0 }   H^{p}(\na u) 
			\, H(\nu)  \, w(x) \, d \cH^{N-1} \right]^{\frac{1}{p'} }  .
		\end{split}
	\end{equation*}
As a result, exploiting \eqref{eq:Gauss-Green} (which also
holds true for $t=0$), we conclude that
	\begin{equation*}
		\int_{\Ga_0}   H^{p}(\na u) \, H(\nu)  \, w(x) \, d \cH^{N-1} \ge \frac{1}{P_{w, H} ( \Om ; \Si  )^{\frac{1}{p-1} }} \, \left[     \int_{ \Si \cap \Om } f(u) \, w(x) \, d \cH^{N} \right]^{ p' } .
	\end{equation*}
	The last inequality, \eqref{eq:usiamo omega ball}
	and \eqref{eq:weighted anisotropic Pohozaev}
	entail that
	\begin{equation}\label{eq:disuguaglianza da destra}
		D \int_{ \Si \cap \Om } F(u) \, w(x) \, d \cH^{N} +  \frac{p-D}{p} \int_{ \Si \cap \Om } u \, f(u) \, w(x) \, d \cH^{N} \ge  \frac{R}{p' \,P_{w,H}(\Om; \Si)^{\frac{1}{p-1} }} \, \left[     \int_{ \Si \cap \Om } f(u) \, w(x) d \cH^{N} \right]^{ p' } .
	\end{equation}
	By recalling \eqref{eq:optimal constant anisotropic isoperimetric} and	
	 \eqref{eq:relazione volume perimetro}, the straightforward computation
	$$
	\frac{D}{p' c^{ p' }}  w \left( \Si \cap \Om \right)^{ \frac{p-D}{D(p-1)} } =  \frac{R}{p' \,P_{w,H}(\Om ; \Si )^{\frac{1}{p-1} }}
	$$
	shows that \eqref{eq:disuguaglianza da sinistra} and \eqref{eq:disuguaglianza da destra} are opposite inequalities. Thus, they hold with the equality sign, and hence, for a.e. $t \in \left( 0,M \right)$,  $\left\lbrace u>t \right\rbrace$ is a is a minimizer of \eqref{eq:weighted anisotropic isoperimetric in cones} and $H( \na u )$ is constant on $ \left\lbrace u=t \right\rbrace$.
\end{proof}

\begin{lem}\label{lem:conclusion in ball}
Suppose that the assumptions of Theorem \ref{thm:MAIN GENERAL} are satisfied. Up to rotations, we
write~$\Si= \RR^k \times \tilde{\Si}$, where $0 \le k \le N$ and $\tilde{\Si} \subset \RR^{N-k}$ is an open convex cone containing no lines.
	Assume that for a.e. $t \in (0,M)$ the following two conditions are verified:
	\begin{itemize}
\item[(i)] $\left\lbrace u >t \right\rbrace = \Si \cap B_{ \rho(t)}(x(t) ) $ where $B_{ \rho(t)}(x(t))$ is the Wulff ball of radius $\rho(t) \ge 0$ centered at the point $x(t) \in \RR^k \times \left\lbrace 0_{\RR^{N-k}} \right\rbrace $, with $\rho(t)$ and $x(t)$ depending on $t$; 
	\item[(ii)] $H( \na u)$ is constant on $\left\lbrace u = t \right\rbrace$.
	\end{itemize}
Then, $u$ is a radially symmetric and radially nonincreasing function. Moreover,
		\begin{equation*}
			u \text{ is radially strictly decreasing on } \left\lbrace 0 < u < M \right\rbrace ,
		\end{equation*}
		and
		$$
		\left\lbrace 0 < u < M \right\rbrace \quad \text{is a Wulff annulus or a punctured Wulff ball centered at } 0 \text{ intersected with } \Si.
		$$
\end{lem}

\begin{proof}
	Let us consider the case $k=N$. We follow the ideas used in \cite[Lemma 6]{Se} in the isotropic setting.
		By reasoning as in \cite[Lemma 6]{Se}, we easily find that
	$$\left\lbrace u>t \right\rbrace = B_{ \rho(t)}(x(t)) \quad \text{ for every } t \in (0,M), $$
	where $B_{ \rho(t)}(x(t))$ denotes the Wulff ball centered at $x(t)$ with radius $\rho(t)$.
	Hence, by continuity also (ii) holds for every $t \in (0,M)$, that is, $H( \na u)$ is constant on $\pa B_{ \rho(t)}(x(t))$ for every $t \in (0,M)$.
	By \eqref{eq:Gauss-Green}, we thus find that
	$$
	P_{w,H} (\pa B_{ \rho(t)}(x(t))) H \left( \na u(\pa B_{ \rho(t)}(x(t)))  \right)^{p-1} = \int_{ B_{ \rho(t)}(x(t)) } f(u) \, w(x) \, d \cH^{N}.
	$$
	Notice that, since $f \ge 0$, by the maximum principle we cannot have $f \equiv 0$ on $\left\lbrace u > t \right\rbrace$ for $t \in (0,M)$. Hence, $\na u$ does not vanish in $\left\lbrace 0 < u < M \right\rbrace$.
	
	Thus, $u$ is a $C^1$ function whose gradient never vanishes in $\left\lbrace u < M \right\rbrace$, and hence $\widetilde{\mu}(t):= \cH^{N} ( \left\lbrace u >t \right\rbrace) $ is locally Lipschitz in $(0,M)$.
	Therefore, also $\rho(t)= \left( \widetilde{\mu}(t)/\cH^{N}(B) \right)^{1/N}$ is locally Lipschitz, being $\cH^{N}(B)$ the volume of the unitary Wulff ball.
	
We observe that $B_{ \rho(t)}(x(t)) = \left\lbrace u \ge t \right\rbrace \supset \left\lbrace u \ge s \right\rbrace = B_{ \rho(s)}(x(s)) $ for $t<s$, and therefore
$$x(s) + \rho(s) \frac{x(s) -x(t)}{H_0(x(s) -x(t))} \in B_{ \rho(s)}(x(s))\subset B_{ \rho(t)}(x(t)) .$$
As a consequence,
	\begin{equation*}
		\begin{split}
			\rho(t) & \ge H_0 \left( x(s) + \rho(s) \frac{x(s) -x(t)}{H_0(x(s) -x(t))} -   x(t) \right)
			\\
			& = H_0 \left( \frac{ x(s)-x(t)}{H_0(x(s) -x(t))} \left(H_0(x(s) -x(t)) + \rho(s) \right) \right)
			\\
			& = H_0(x(s) -x(t)) +\rho(s) ,
		\end{split}
	\end{equation*}
	where in the last identity we used the homogeneity of $H_0$, and accordingly
	\begin{equation}\label{eq:x locally Lipschitz}
		H_0( x(t) - x(s)) \le \rho(t) - \rho(s) \quad \text{ for } t<s ,
	\end{equation}
	that gives that $x(t)$ is also locally Lipschitz.
	
Now suppose by contradiction that $u$ is not radially symmetric. Then $x(t)$ would not be identically constant in $(0,M)$ and hence we could find some $t_0 \in (0,M) $ such that the velocity vector $x' (t_0 )$ would exist and be nonzero. Set
	$$\ol{\xi}:= \frac{x'( t_0 )}{ H_0 \left( x'( t_0 ) \right) } ,\qquad
	P(t) := x(t) + \rho(t) \ol{\xi} 
	\quad {\mbox{and}} \quad
	Q(t) := x(t) - \rho(t) \ol{\xi} .
	$$
	By definition, $P(t)$ and $Q(t)$ belong to $\pa \left\lbrace u>t \right\rbrace $ for every $t$, and hence we have
	\begin{equation}\label{eq:uP=uQ conto in lemma conclusivo}
		u \left( P(t) \right) \equiv u \left( Q(t) \right) \equiv t \quad \text{ for every } t ,
	\end{equation}
	\begin{equation}\label{eq:conclusive lemma gradienti P Q su livello}
		\na u \left( P(t_0) \right) = - | \na u \left( P(t_0) \right) | \, \nu \left( P(t_0) \right) 
		\quad \text{ and } \quad
		\na u \left( Q(t_0) \right) = - | \na u \left( Q(t_0) \right) | \, \nu \left( Q(t_0) \right) ,
	\end{equation}
	where $\nu$ denotes the exterior unit normal vector field to
	$\left\lbrace u >t_0 \right\rbrace = B_{ \rho(t_0)}(x(t_0)) $.
	Moreover, since $\ol{ \xi} \in B= \left\lbrace x \in \RR^N \, : \, H_0 ( x )=1 \right\rbrace$, by \eqref{eq:norma H def dual} it holds that
	\begin{equation}\label{eq:formula Wulff ball xi prodotto scalare con normale}
	\langle \nu_{B} (\ol{\xi} ), \ol{ \xi } \rangle = H( \nu_B ( \ol{\xi} ) ) .
	\end{equation}
	We also observe that
	\begin{equation}\label{eq:relazioni normali e ol xi}
	\nu \left( P(t_0) \right) = \nu_{B} (\ol{\xi} ) = - \nu_{B} ( - \ol{\xi} ) = - \nu \left( Q(t_0) \right) .
	\end{equation}
	To check this, we let $B_r= \left\lbrace H_0(x) < r \right\rbrace$ be the Wulff ball of radius $r$ centered at the origin, and set $B:=B_1$.
				For any $x \in \pa B_r$ it holds that
				$$
				\nu_{B_r} (x) = \frac{\na_x H_0(x)}{ | \na_x H_0(x) |} 
				\quad \text{and} \quad
				\nu_{B_r} (- x) = \frac{\na_x H_0(-x)}{ | \na_x H_0(-x) |}.
				$$
				Also, by the homogeneity of $H_0$ we have that 
				$$
				\na_x H_0 (t x ) = \mathrm{sgn}(t) \na_x H_0 (x) \quad \text{for all } \, t \neq 0, \, x \neq 0 
				$$
				and consequently
				\begin{equation}\label{eq:aggiunta a conclusivelemma conto normali}
				\nu_{B_r} (r \, \xi ) = \nu_{B} ( \xi )= - \nu_{B} (- \xi ) = \nu_{B_r} ( - r \, \xi )  \quad \text{ for any } \xi \in \pa B , \, r \in \RR^+ .
				\end{equation}
Additionally, taking into account the translation invariance of the normals, we obtain
				$$
				\nu_{B_r (x)} \left( x + r \, \xi \right) = \nu_{B_r} \left( r \, \xi \right) \quad \text{ for any } x \in \RR^N, \, r \in \RR^+, \, \xi \in \pa B .
				$$			
				Combining this information and~\eqref{eq:aggiunta a conclusivelemma conto normali} we obtain~\eqref{eq:relazioni normali e ol xi}, as desired.
				
				Thus, from~\eqref{eq:formula Wulff ball xi prodotto scalare con normale} and \eqref{eq:relazioni normali e ol xi}
		we deduce that
		$$
		\langle  \nu \left( P(t_0) \right) , \ol{\xi} \rangle =  H \left(\nu \left( P(t_0) \right)  \right)  
		\quad \text{ and } \quad
		\langle \nu \left( Q(t_0) \right) , \ol{\xi} \rangle = -  H \left( \nu \left( Q(t_0) \right) \right) .
		$$ 
	From this and \eqref{eq:conclusive lemma gradienti P Q su livello}, it follows that
	\begin{equation}\label{eq:olxi conto per lemma conclusivo}
		\langle \na u \left( P(t_0) \right) , \ol{ \xi } \rangle =   - H \left( \na u \left( P(t_0) \right)  \right) \quad \text{ and } \quad 
		\langle \na u \left( Q(t_0) \right) , \ol{ \xi } \rangle =   H \left( \na u \left( Q(t_0) \right)  \right) .
	\end{equation}
	
	By using \eqref{eq:uP=uQ conto in lemma conclusivo} and \eqref{eq:olxi conto per lemma conclusivo}, we compute
	$$
	1= \frac{d}{dt} u \left( P(t) \right) |_{t=t_0} = \langle \na u  \left( P(t_0) \right) , \left( H_0 ( x'(t_0) ) + \rho'(t_0) \right) \ol{ \xi } \rangle = - H \left( \na u \left( P(t_0) \right)  \right) \left( H_0 ( x'(t_0) ) + \rho'(t_0) \right)
	$$ and
	$$
	1= \frac{d}{dt} u \left( Q(t) \right) |_{t=t_0} = \langle \na u  \left( Q(t_0) \right) , \left( H_0 ( x'(t_0) ) - \rho'(t_0) \right) \ol{ \xi } \rangle =  H \left( \na u \left( Q(t_0) \right)  \right) \left( H_0 ( x'(t_0) ) - \rho'(t_0) \right) .
	$$
	
	But by assumption (ii), we have that $ H \left( \na u \left( P(t_0) \right) \right) =  H \left( \na u \left( Q(t_0) \right) \right) $. This leads to $H_0 ( x'(t_0) )=0$, which is a contradiction. 
	
	Thus, $u$ is radially symmetric, i.e., $u = u(r)$, with $r=H_0(x)$. Since we showed that $\na u$ does not vanish on $\left\lbrace 0<u<M \right\rbrace$, we thus have $\pa_r u <0$ in this open ring. The set $\left\lbrace u=M \right\rbrace$ could be a point, but also a closed ball (as it happens, e.g., in the example described in \cite[pag. 1895]{Se}).
This proves the result when $k=N$.
	
	By direct inspection and recalling that $u$ is $C^1$ up to $\Ga_1 \setminus \left\lbrace 0 \right\rbrace $, it is easy to check that the same proof remains valid when $1 \le k \le N-1$. Notice that in this case we have that $\ol{\xi} \in \RR^{k} \times \left\lbrace 0_{\RR^{N-k}} \right\rbrace $. 
	
	The case $k=0$ is trivial.
\end{proof}

We have now all the ingredients to complete the 

\begin{proof}[Proof of Theorems \ref{thm:cones}, \ref{thm:MAIN unweighted in Wulff ball} and \ref{thm:weighted cones}]
Theorems \ref{thm:cones} and \ref{thm:MAIN unweighted in Wulff ball} easily follow by putting together Theorems \ref{thm:MAIN GENERAL}, \ref{thm:characterization unweighted anisotropic}, and Lemma \ref{lem:conclusion in ball}.

Theorem \ref{thm:weighted cones} easily follows by putting together Theorems \ref{thm:MAIN GENERAL}, \ref{thm:characterization isotropic weighted ISOperimetric}, and Lemma \ref{lem:conclusion in ball}.
\end{proof}

\section*{Acknowledgements}
The authors are members of INdAM/GNAMPA and AustMS and
are supported by the Australian Research Council
Discovery Project DP170104880 NEW ``Nonlocal Equations at Work''.
The first author is supported by
the Australian Research Council DECRA DE180100957
``PDEs, free boundaries and applications''. 
The second and third authors are supported by
the Australian Laureate Fellowship
FL190100081
``Minimal surfaces, free boundaries and partial differential equations''.

The authors thank Xavier Cabr\'e for bringing up to their attention the reference~\cite{Se}. 
The authors also thank Xavier Ros-Oton and Joaquim Serra for their useful
comments on a preliminary version of this paper, Giulio Ciraolo for the
helpful correspondence about the forthcoming article~\cite{ACF}, and Emanuel Indrei for kindly pointing out his recent results in~\cite{I}.

\vfill

\end{document}